\documentclass{siamart190516}
\usepackage{amssymb,amsmath,amsfonts,verbatim,cite,subcaption,tikz}

\usepackage[section]{placeins}
\usepackage[algo2e,ruled,linesnumbered]{algorithm2e}

\newcommand{\ie}{i.e., }

\newcommand{\range}{\mathop\mathrm{range}}

\newcommand{\Ttran}{\mathsf{T}}
\newcommand{\T}{\mathsf{T}}

\newcommand{\fro}{\mathsf{F}}
\newcommand{\defi}{:=}
\newcommand{\van}{\mathsf{Van}}
\newcommand{\normml}{\left\vert\kern-0.25ex\left\vert\kern-0.25ex\left\vert}  
\newcommand{\normmr}{\right\vert\kern-0.25ex\right\vert\kern-0.25ex\right\vert}
\newcommand{\order}{\mathcal{O}}

\newcommand{\R}{\mathbb{R}}
\newcommand{\rk}{\mathrm{rank}}
\newcommand{\Ta}{\Xi}
\newcommand{\Tb}{\Gamma}
\newcommand{\cheb}{\mathsf{Cheb}}
\newcommand{\gap}{\mathsf{relgap}}
\newcommand{\old}{\mathrm{old}}
\newcommand{\new}{\mathrm{new}}
\newcommand{\mpr}{\mathbf{u}}

\newcommand{\figsize}{0.7\textwidth}

\providecommand{\diagm}{\mathrm{diag}}
\providecommand{\diagM}[1]{\mathrm{diag}(#1)}
\providecommand{\abs}[1]{\lvert#1\rvert}
\providecommand{\norm}[1]{\lVert#1\rVert}

\providecommand{\bignorm}[1]{\bigl\lVert#1\bigr\rVert}

\crefname{Assumption}{Assumption}{Assumptions}
\crefname{algocf}{Algorithm}{Algorithms}
\newsiamremark{remark}{Remark}
\newsiamremark{example}{Example}
\newsiamthm{assumption}{Assumption}
 
\newcommand{\edit}[1]{\textcolor{black}{#1}} 
\newcommand{\editr}[1]{\textcolor{black}{#1}} 
\newcommand{\redit}[1]{\textcolor{black}{#1}} 

\begin{document}
    \title{{\redit{On a}} randomized small-block Lanczos method for large-scale null space computations
    }
    \author{
        Daniel~Kressner\thanks{Institute of Mathematics, EPF Lausanne, 1015 Lausanne, Switzerland (\href{mailto:daniel.kressner@epfl.ch}{daniel.kressner@epfl.ch} and \href{mailto:nian.shao@epfl.ch}{nian.shao@epfl.ch})}
        \and
        Nian~Shao\footnotemark[1]
        }
    \headers{Randomized Lanczos for null space computations}{Daniel~Kressner and Nian~Shao}
    \maketitle   
        \begin{abstract}	
 Computing the null space of a large sparse matrix $A$ is a challenging computational problem, especially if the nullity---the dimension of the null space---is \edit{not small}. When \redit{applying a block Lanczos method to $A^\mathsf{T} A$} for this purpose, conventional wisdom suggests to use a block size $d$ that is not smaller than the nullity. In this work, we show how randomness can be utilized to allow for smaller $d$ without sacrificing convergence or reliability. Even $d = 1$, corresponding to the standard single-vector Lanczos method, becomes a safe choice. This is achieved by using a small random diagonal perturbation, which moves the zero eigenvalues of $A^\mathsf{T} A$ away from each other, and a random initial guess. We analyze the effect of the perturbation on the attainable quality of the null space and derive convergence results that establish robust convergence for $d=1$. As demonstrated by our numerical experiments, a smaller block size combined with restarting and partial reorthogonalization results in reduced memory requirements and computational effort. It also allows for the incremental computation of the null space, without requiring a priori knowledge of the nullity. Our algorithm is best suited for situations when the nullity of $A$ is \redit{moderate}.
        \end{abstract}
        
        \begin{keywords}
            Sparse matrix, null space, Krylov subspace method, randomized algorithm
        \end{keywords} 
        \begin{AMS}
            65F15, 65F50, 68W20
        \end{AMS}
\maketitle

\section{Introduction}
This work is concerned with computing the null space of a large sparse matrix $A\in \R^{m\times n}$, a problem
that finds numerous applications. For example, the connected components of an undirected graph can be identified by the null space of the graph Laplacian~\cite{Merris1994}. In computational topology, the null space of cohomology matrices reflects the generators of the cohomology groups. The nullity helps determine the Betti numbers of a topological space~\cite{Edelsbrunner2022}. Furthermore, null space calculations play a crucial role in various engineering problems, such as system dimensionality reduction~\cite{Ye2006}, finite element models~\cite{Shklarski2009}, linear equality constrained state estimation~\cite{Hewett2010}, grasp control~\cite{Platt2010}, computer vision~\cite{Zhang2016a}, and multivariate polynomial root finding~\cite{Batselier2014}.

In several of these applications, the dimension of the null space is not necessarily small and can easily reach a few hundred or even thousands.
For small- to medium-sized matrices, the singular value decomposition (SVD)~\cite{Golub2013} or a suitably modified QR decomposition~\cite{Foster1986} are the methods of choice for 
computing null spaces. For large-sized matrices, these dense methods become too expensive. There are several structured null space solvers to deal with specific large-scale situations, such as~\cite{Berry1985} for banded matrices, \cite{Gotsman2008} for matrices with small nullity and amenable to sparse LU factorization with partial pivoting, \cite{Foster2013} for sparse matrices amenable to a rank-revealing sparse QR factorization~\cite{Davis2011a}, and \cite{Park2023} for a very tall and skinny matrix.

When $A$ is large and has no particular structure beyond being sparse, one needs to resort to Krylov subspace methods, which utilize matrix-vector products with $A$ and its transpose $A^\T$. Let $N$ denote the null space dimension of $A$. When $N = 1$, standard choices include the Lanczos method for computing the smallest eigenvalue and eigenvector of $A^{\Ttran}A$ or the mathematically equivalent, but numerically more reliable
Golub--Kahan bidiagonalization directly applied to $A$~\cite[Chapter~10]{Golub2013}. When $N>1$, the Lanczos method (with a single starting vector) becomes unreliable and prone to miss components of the null space. In fact, in exact arithmetic only a one-dimensional subspace of the null space can be extracted via Lanczos~\cite[Section~12.2]{Parlett1998}. Block Lanczos methods address this issue by utilizing a block of $d>1$ starting vectors instead of a single starting vector \cite[Section~13.10]{Parlett1998}. 
Common wisdom (see \cite[Section~13.12]{Parlett1998} for examples) and existing convergence analyses all suggest to choose $d$ not smaller than the size of the eigenvalue cluster of interest. Applied to null space computation this would require $d \ge N$ and in exact arithmetic this condition is indeed necessary. The presence of round-off error blurs the situation.
\editr{
As an example, consider the $10064\times 10064$ diagonal matrix 
\begin{equation} \label{eq:matrixA}
    A = \diagm \Bigl(0,\dotsc,0,1+\frac{1}{10000},1+\frac{2}{10000},\dotsc,2\Bigr),
\end{equation}
that is, the first $64$ diagonal entries are zero.} \redit{Because this matrix is symmetric positive semi-definite, one can apply block Lanczos directly to $A$} in order to compute its null space. \redit{When} choosing a block of $d$ (Gaussian) random starting vectors, one would expect to obtain no more than $d$ (approximate) zero eigenvalues and corresponding eigenvectors. However, when executing block Lanczos \editr{with full reorthogonalization} in double precision, \editr{limiting it to $1024$ matrix-vector multiplications with $A$ in total} and declaring eigenvalues less than $10^{-4}$ as zero, we obtained the results reported in~\cref{tabintro}. \editr{
\begin{table}[htbp]
    \centering
    \caption{Null space dimension obtained when applying block Lanczos with $d$ random starting vectors \editr{for 1024 matrix-vector multiplications} to the matrix from~\cref{eq:matrixA}.}
    \label{tabintro}
    \editr{
    \begin{tabular}{|c|c|c|c|c|c|c|c|}
        \hline
        Block size $d$ & 1 & 2 & 4 & 8 & 16 & 32 & 64  \\ \hline 
        Dimension &   45 & 46 & 48 & 48 & 48 & 64 & 64\\ \hline 
    \end{tabular}}
\end{table}
While the null space dimension estimated by block Lanczos with small block size $d$ is erratic and often smaller than $64$, it is always larger than $d$. In fact, already for $d = 32$, the correct null space dimension is detected.} Experiments with other matrices lead to similar findings, suggesting that round-off error breaks some of the curse incurred by eigenvalue clusters, but it does not fully address the convergence issues of single-vector or small-block Lanczos methods either.

The situation described above is reminiscent of recent work by Meyer, Musco and Musco~\cite{Meyer2023} on the single-vector Lanczos method for low-rank approximation. Taking the randomness of the starting vector into account they establish a convergence result that still requires the (larger) singular values to be distinct but the dependence of the complexity on the gaps between eigenvalues is very mild, in fact logarithmic. Although single-vector or small-block Krylov subspace methods do not satisfy gap-independent convergence bounds~\cite{Li2015,Saad1980,Musco2015,Yuan2018}, the theory in~\cite{Meyer2023} suggests that small random perturbations of $A$ can easily break the adverse role of repeated singular values. As we have seen, perturbations due to round-off are not sufficient to achieve this effect, but (slightly larger) random perturbations will do, thanks to eigenvalue repulsion results.

\edit{Inspired by~\cite{Meyer2023}, we propose a randomized small-block Lanczos method for null space calculation. 
Our method assumes that there is a significant gap between the zero and nonzero singular values. We first modify the matrix $A^{\Ttran}A$ by adding a small random diagonal perturbation, which separates its zero eigenvalues. This separation enables us to use the small-block Lanczos method to capture an approximation for the \emph{whole} null space basis of $A$.} 
Compared to the task of low-rank approximation, there are two major differences. First, in low-rank approximation, the largest few singular values are usually unknown and, in most applications, they do not form tight large clusters. In the context of null spaces, the desired singular values form one large cluster of zeros. Second, while convergence to the relevant invariant subspace is not necessary for low-rank approximation~\cite{Drineas2019}, such convergence is imperative to obtain a good null space approximation.

The rest of this paper is organized as follows.
\edit{In \Cref{sec:perturbationtheory}, we analyze the impact of the random perturbation on $A^{\Ttran}A$.} 
    In particular, \cref{thmgap} provides a new eigenvalue repulsion result for the perturbed zero eigenvalues of $A^{\Ttran}A$. Our result differs from an existing result~\cite[Lemma 12]{Meyer2023} in two aspects: (1) By focusing on the repulsion of zero eigenvalues rather than the entire spectrum, we establish $\order(\epsilon)$ gaps instead of $\order(\epsilon^{2})$ gaps. (2) By considering relative eigenvalue gaps, we establish \edit{a} lower bound independent of the norm of $A$. At the same time, the perturbation incurs a limit on the attainable accuracy of the null space approximation, and we quantify this effect by a perturbation analysis.

    \Cref{Sec:ConvAnal} performs a self-contained convergence analysis of the randomized single-vector Lanczos method \edit{(applied to the perturbed matrix)}. 
    While our analysis shares similarities with \cite{Meyer2023}, we achieve more refined results by focusing on $d = 1$ 
    and considering the convergence of individual Ritz vectors. This analysis cumulates in~\cref{thmLan,thmmain}, establishing bounds on the null space approximation, whose sharpness is verified by the numerical results. 

    In \Cref{sec:implementation}, we consider several practical improvements, 
    including preconditioning, restarting, and partial reorthogonalization\redit{, culminating into \cref{algoad}, our main algorithm. Note that, due to restarting, the convergence analysis of
    \Cref{Sec:ConvAnal} does not apply to this algorithm.} 
    We also stress that our algorithm is conceptually quite close to existing restarted eigenvalue solvers, such as ARPACK~\cite{Lehoucq1998}, with the notable exception of the partial reorthogonalization strategy (which becomes important when dealing with larger null space dimensions). 
    
    Finally, \Cref{Sec:NumExp} contains several numerical results for applications, such as the computation of connected graph components and cohomology computation. These numerical experiments not only validate our theoretical findings regarding the randomized single-vector Lanczos method but also highlight the efficiency of the small-block Lanczos method. Notably, it highlights when our method is preferable compared to other null space solvers, particularly when the nullity is substantial and memory constraints are a limiting factor.
    
\subsection{Assumptions and notation}

Throughout this paper we assume that $A \in \R^{m\times n}$ satisfies $m\ge n$ and has $N\ge 1$ zero singular values.
We let $\underline{\sigma}$ denote the smallest positive singular value of $A$. We assume that there is a (significant) gap between $\underline{\sigma}$ and $0$, which allows us to pick $\epsilon$ such that $0 < \epsilon<\underline{\sigma}^{2}$. Such an assumption on the separation between zero and non-zero singular values is reasonable for matrices that have an exact null space in the absence of round-off error. \redit{Our method is not suitable when the singular values decay gradually to zero; see~\cite{Beckermann2019} for examples of such matrices.}
\begin{remark}
    For null space computation, the cases $m\geq n$ and $m<n$ are fundamentally different because $m<n$ implies that $A$ has a null space of dimension at least $n-m$. \edit{Our algorithm can also be applied to this situation and our analysis carries over. However, when $n$ is much larger than $m$ (say, $n=2m$), its efficiency is compromised by the large null space dimension. In these situations,}
    one needs to resort to dense methods like the SVD, \edit{unless there is additional structure that can be exploited}.
\end{remark}

Throughout the remainder of the paper, we use $\norm{\cdot}$ to denote the Euclidean norm of a vector and the spectral norm of a matrix.
\subsection{Basic algorithm}
In accordance with the first step, we define
\begin{equation} \label{eq:defB}
    B := A^{\Ttran}A+\epsilon D,
\end{equation}  
where $D$ is diagonal with nonnegative diagonal entries and $\epsilon>0$ controls the strength of the perturbation.
The initial $\Omega\in\R^{n\times d}$ is chosen as a Gaussian random matrix, \ie all entries are i.i.d. (independent and identically distributed) from the standard normal distribution $\mathcal{N}(0,1)$.
\edit{In the second step, we perform block Lanczos with $\ell$ iterations to compute an orthonormal basis $Z_{\ell}$ for the Krylov subspace\footnote{\edit{Note that our definition of a block Krylov subspace is different from some of the literature~\cite{Gutknecht2009}, but consistent with~\cite{Stewart2001,Meyer2023}.}} 
    \begin{equation*}
        \mathcal{K}_{\ell}(B,\Omega) \defi \range[\Omega,B\Omega,\dotsc,B^{\ell-1}\Omega].
    \end{equation*}
    To simplify the discussion, we will always assume that $\mathcal{K}_{\ell}(B,\Omega)$ has full rank and, therefore, 
    $Z_{\ell}$ is of size $n\times \ell d$. \redit{For this subspace to contain any approximation to the $N$-dimensional null space, $\ell d$ must be at least $N$.} Note that the block Lanczos method also \redit{returns} the compression $T_{\ell}=Z_{\ell}^{\Ttran}BZ_{\ell}\in\R^{\ell d\times\ell d}$.
    We use a Rayleigh--Ritz process to extract an approximate null space from $\mathcal{K}_{\ell}(B,\Omega)$. More specifically, the spectral decomposition of $T_{\ell}$ is computed to 
    obtain the smallest $N$ eigenvalues of $T_{\ell}$ (called Ritz values) along with an orthonormal basis $V_Z$ of the corresponding eigenvectors. Then $V=Z_{\ell}V_{Z}$ contains the corresponding Ritz vectors, which tend to approximate eigenvectors belonging to the smallest eigenvalues of $B$. This makes it natural to choose $V$ as an approximate null space basis of $A$, leading to~\cref{algo}.
\begin{algorithm2e}[htbp]
	\caption{Randomized small-block Lanczos for null space computation} \label{algo}
	\KwIn{Matrix $A\in\R^{m \times n}$ with $m\ge n$ and $N$ zero singular values.  Perturbation parameter $\epsilon>0$. Block size $d$. Number of iterations $\ell$ \redit{such that $\ell d \ge N$}.}
    \KwOut{Orthonormal basis $V$ approximating null space of $A$.}
    Set $B=A^{\Ttran}A+\epsilon D$, with a diagonal matrix $D$ having diagonal entries uniformly i.i.d. in $[0,1]$\;
    Draw Gaussian random matrix $\Omega\in\R^{n\times d}$\;
    \edit{Perform the block Lanczos method for $\ell$ iterations to compute an orthonormal basis $Z_{\ell}\in\R^{n\times \ell d}$ for $\mathcal{K}_{\ell}(B,\Omega)$ and the block tridiagonal matrix $T_{\ell}=Z_{\ell}^{\Ttran}BZ_{\ell}\in\R^{\ell d\times \ell d}$\;}
    Compute an orthonormal basis \edit{$V_{Z}\in\R^{\ell d\times N}$} \edit{of eigenvectors} associated with $N$ smallest eigenvalues of $T_{\ell}$\; 
   \Return{$V=Z_{\ell}V_{Z}$}\;
\end{algorithm2e}}

\edit{Note that the matrix $B$ is never formed explicitly in~\cref{algo}, but it is only accessed through matrix-vector products, which are carried out through matrix-vector products with the original matrix $A$ and $A^\Ttran$: $Bx = A^{\Ttran}(Ax) + \epsilon (Dx)$.}
\begin{remark}
\edit{As explained above,}
Line~\edit{4} of \cref{algo} extracts \edit{a} null space approximation \edit{via} computing the spectral decomposition of $T_{\ell} = Z_{\ell}^\T B Z_{\ell}$.
Alternatively, one could remove the perturbation and compute the spectral decomposition of $Z_{\ell}^\T A^\T A Z_{\ell}$ or, equivalently, the SVD of $AZ_{\ell}$. Preliminary numerical experiments indicate that this only yields minor accuracy improvements, at best, not justifying the extra computational effort needed to perform the alternative procedure.  
\end{remark}

\section{Perturbation analysis} \label{sec:perturbationtheory}

The purpose of this section is to quantify the impact
of the (random) perturbation $\epsilon D$ introduced by \cref{algo}.

\subsection{Perturbation and repulsion of zero eigenvalues}

We start by considering the impact of the perturbation on the $N$ zero eigenvalues of $A^\T A$.
Suppose that $0<\epsilon\|D\| <\underline{\sigma}^{2}$,
where \edit{we recall that} $\underline{\sigma}$ denotes the smallest nonzero singular value of $A$. Then the zero and nonzero eigenvalues of $A^\T A$ will not be mixed by the perturbation $\epsilon D$. Indeed, using that $D$ is positive semi-definite, it follows from 
Weyl's inequality~\cite[Theorem 4.3.1]{Horn2012} that the eigenvalues $\lambda_1,\dotsc, \lambda_n$ of $B$ satisfy
\begin{equation} \label{eq:propevB}
    0\leq \lambda_{1}\leq \dotsb\leq \lambda_{N}\leq \epsilon \|D\| < \underline{\sigma}^{2}  \leq \lambda_{N+1}\leq \dotsb\leq \lambda_{n}. 
\end{equation}

\cref{thmgap} below shows that it is unlikely that the perturbed zero eigenvalues $\lambda_{1}, \dotsc, \lambda_{N}$ remain very close to each other when $D$ is randomly chosen. To establish this eigenvalue repulsion result, we follow the proof of~\cite[Lemma 12]{Meyer2023} and make use of the following proposition.
 \begin{proposition}[Equation~(1.11) in~\cite{Aizenman2017} attributed to~\cite{Minami1996}]
    \label{repulsion}
    Let $B_0 \in\R^{n\times n}$ be symmetric, and let $D\in\R^{n\times n}$ be a random diagonal matrix with the diagonal entries drawn i.i.d. from a distribution with bounded probability density function $\rho(\cdot)$. Then for any bounded interval $\mathcal{I}\subset \R$ and $\epsilon > 0$, the inequality
    \begin{equation*}
        \Pr(B_0 + \epsilon D \text{ has at least $2$ eigenvalues in }\mathcal{I})\leq C_0 (\epsilon^{-1} \norm{\rho}_{\infty}\abs{\mathcal{I}}n)^{2},
    \end{equation*}
    holds for some universal constant $C_0$, where $\abs{\mathcal{I}}$ is the length of $\mathcal{I}$.
\end{proposition}

\Cref{repulsion} is a probabilistic result: For an arbitrary but fixed interval $\mathcal{I}$, it is very unlikely that more than one eigenvalue ends up in that interval, provided that the interval is sufficiently small, $\abs{\mathcal{I}} \ll \epsilon /n$. In particular, for an eigenvalue $\lambda_i$ of $B_0 + \epsilon D$ it is very unlikely that $\lambda_{i-1}$ or $\lambda_{i+1}$ are very close to $\lambda_i$. \Cref{thmgap} below is a consequence of these considerations. In contrast, when $\abs{\mathcal{I}} \gtrsim  \epsilon /n$, the probability bound of \Cref{repulsion} may be larger than $1$ and no conclusion can be drawn.
\begin{theorem}
    \label{thmgap}
    Consider the matrix $B$ defined in~\eqref{eq:defB} and suppose that the diagonal entries of $D$ are drawn i.i.d. from the uniform distribution on $[0,1]$.
    If $0<\epsilon<\underline{\sigma}^{2}$, the inequality
    \begin{equation*}
        \min_{1\leq i\leq N-1}\abs{\lambda_{i+1}-\lambda_{i}} \geq \frac{\delta\epsilon}{Cn^{2}},
    \end{equation*}
    holds with probability at least $1-\delta$ for any $0<\delta<1$, where
$C$ is a universal constant.
\end{theorem}
\begin{proof}
Because of $\|D\| \le 1$, it follows from~\eqref{eq:propevB} that $\lambda_{1},\dotsc,\lambda_{N}$ are contained in the interval $[0,\epsilon]$.
We cover this interval with overlapping (smaller) intervals $\mathcal{I}_{j}=[j\gamma,(j+2)\gamma]$ for $j = 0,\ldots,\epsilon/\gamma-2$, where---for the moment---$\gamma$ is an arbitrary number such that $0<\gamma\leq \epsilon/2$ and $\epsilon/\gamma$ is an integer. 
If two neighboring eigenvalues satisfy $\lambda_{i+1}-\lambda_{i}\leq\gamma$, then there clearly exists an interval $\mathcal{I}_{j}$ such that $\lambda_{i},\lambda_{i+1}\in\mathcal{I}_{j}$. Thus, 
    \begin{eqnarray*}
        \Pr \{ \exists i\colon \lambda_{i+1}-\lambda_{i}\leq \gamma\} &\leq & 
        \Pr \{\exists i,\, j\colon \lambda_{i},\lambda_{i+1}\in\mathcal{I}_{j}\}\leq \sum_{j=0}^{\epsilon/\gamma-2}\Pr \{\exists i\colon \lambda_{i},\lambda_{i+1}\in\mathcal{I}_{j}\} \\
        &\le & \sum_{j=0}^{\epsilon/\gamma-2} 4 C_0 n^{2} \epsilon^{-2} \gamma^{2}= 4C_0 n^{2}\gamma \epsilon^{-1} \big(1-\gamma\epsilon^{-1} \big) \le \frac{4C_0 n^{2}}{\epsilon/\gamma+1},
    \end{eqnarray*}
    where we used the union bound and~\cref{repulsion} in the second and third inequalities, respectively. The last inequality uses $x(1-x) = x (1-x^2) / (1+x) \le x / (1+x)$ with $x = \gamma \epsilon^{-1}$.
    The proof is concluded by setting $C = 4C_0$, $\gamma = \epsilon/\lfloor \frac{Cn^{2}}{\delta}\rfloor\geq \frac{\delta\epsilon}{Cn^{2}}$ and using 
    ${\epsilon}/{\gamma} = \lfloor {Cn^{2}}/{\delta}\rfloor> {Cn^{2}}/{\delta}-1$.
\end{proof}

\subsection{Perturbation of null space residual}

We will bound the impact of perturbations on the null space residual of $A$ using the following eigenvalue perturbation result. 
\begin{lemma}
    \label{propEvalPert}
    Consider symmetric matrices $B_{0}$ and $B \in \R^{n\times n}$ with eigenvalues
    $0\leq \lambda_{1}^{(0)} \leq \dotsb\leq \lambda^{(0)}_{n}$
    and $0\leq \lambda_{1} \leq \dotsb\leq \lambda_{n}$, respectively.
    Let $V_{N}$ be an orthonormal basis for the invariant subspace of $B$ associated with $\lambda_{1}, \dotsc, \lambda_{N}$. Suppose that $B-B_{0}$ is positive semi-definite, then 
    \begin{equation*}
        \norm{V_{N}^{\Ttran}B_{0}V_{N}} \leq \lambda_{N}^{(0)}+ \norm{B-B_{0}}.
    \end{equation*}
    If, additionally, $\norm{B-B_{0}}<\lambda_{N+1}^{(0)}-\lambda_{N}^{(0)}$ then 
    \begin{equation*}
        \norm{V_{N}^{\Ttran}B_{0}V_{N}}\leq \lambda_{N}^{(0)}+\frac{\norm{B-B_{0}}^{2}}{\lambda_{N+1}^{(0)}-\lambda_{N}^{(0)}-\norm{B-B_{0}}}.
    \end{equation*}
\end{lemma}

\begin{proof}
    The first result is a direct consequence of Weyl's inequality:
    \begin{equation*}
        \norm{V_{N}^{\Ttran}B_{0}V_{N}} \leq \norm{V_{N}^{\Ttran}BV_{N}}
        =\lambda_N 
        \leq \lambda_{N}^{(0)}+\norm{B-B_{0}},
    \end{equation*}
    The second result follows from a quadratic residual bound for eigenvalues in~\cite[Theorem~1]{Mathias1998}. To apply this bound, we let $V_\perp$ be such that $[V_N,V_\perp]$ is an $n\times n$ orthogonal matrix and consider the gap
    \begin{equation}
        \label{eq:lowerbound}
        \begin{aligned}
            g &= \lambda_{\min}\big( V_\perp^{\Ttran} B_0 V_\perp \big) - \lambda_{N}^{(0)}
     \ge \lambda_{\min}\big( V_\perp^{\Ttran} B V_\perp \big) - \lambda_{N}^{(0)} - \|B-B_0\|   \\
     &=  \lambda_{N+1} - \lambda_{N}^{(0)} - \|B-B_0\| \ge 
     \lambda_{N+1}^{(0)} - \lambda_{N}^{(0)} -  \|B-B_0\| > 0.
        \end{aligned}
    \end{equation}
    Now, \cite[Theorem~1]{Mathias1998} yields
    \begin{equation*}
        \norm{V_{N}^{\Ttran}B_{0}V_{N}}-\lambda_{N}^{(0)}
        \leq g^{-1} \norm{(I-V_{N}V_{N}^{\Ttran})B_{0}V_{N}}^{2} \le
        g^{-1} \norm{B-B_{0}}^{2},
    \end{equation*}
    where the second inequality uses that $V_{N}$ spans an invariant subspace of $B$. Combined with~\cref{eq:lowerbound}, this proves the desired result.
\end{proof}

\cref{propEvalPert} provides two kinds of perturbation bounds. The first bound does not require a gap but depends linearly on $\norm{B-B_{0}}$, while the second one improves this to quadratic dependence at the expense of requiring a gap between $\lambda_{N}^{(0)}$ and $\lambda_{N+1}^{(0)}$. 
\begin{theorem}
    \label{thmVANg}
    Let $V_{N} \in \R^{n\times N}$ be an orthonormal basis of the invariant subspace associated with the $N$ smallest eigenvalues of $B=A^{\Ttran}A+\epsilon D$ for a symmetric positive semi-definite matrix $D$ with $\norm{D}\leq 1$. 
    \edit{If $0<\epsilon<\underline{\sigma}^{2}$ then
    \begin{equation*}
        \norm{AV_{N}}\leq \epsilon (\underline{\sigma}^{2}-\epsilon)^{-1/2}.
    \end{equation*}
    Furthermore, if $0<\epsilon<\underline{\sigma}^{2}/2$, then $\norm{AV_{N}}\leq \sqrt{\epsilon}$.}

\end{theorem}
\begin{proof}
The result essentially follows from inserting $B_0 = A^{\Ttran} A$ into the second bound of \cref{propEvalPert}.
\end{proof}

\subsection{Principal angles and null space residual}
To study the impact of perturbations on subspaces, we recall the definition and basic properties of principal angles. For later purposes, it will be helpful to admit subspaces of unequal dimension, as in~\cite[Section~6.4.3]{Golub2013} and \cite{Zhu2013}. 

\begin{definition}
    \label{defPA}
    Let the columns of $V_{0}\in\R^{n\times k_{0}}$ and $V_{1}\in\R^{n\times k_{1}}$ form orthonormal bases of subspaces $\mathcal{V}_{0}$ and 
    $\mathcal{V}_{1}$, respectively. Let $k = \min\{k_{0},k_{1}\}$ and let $1 \ge \sigma_{1} \geq  \sigma_{2} \geq \dotsb \geq \sigma_{k} \geq 0$ denote the singular values of $V_{0}^{\Ttran}V_{1}$.
    Then the principal angles $\angle_{t}(V_{0},V_{1}) \in [0,\pi/2]$, $t = 1,\dotsc,k$, between $\mathcal{V}_{0}$ and $\mathcal{V}_{1}$ are defined by
    \begin{equation*}
        \cos \angle_{t}(\mathcal{V}_{0},\mathcal{V}_{1}) = \sigma_{t}.   
    \end{equation*}
\end{definition}

For brevity, we will write $\angle_{t}(\Omega_{0},\Omega_{1})$ instead of $\angle_{t}(\mathcal{V}_{0},\mathcal{V}_{1})$ for arbitrary, not necessarily orthonormal, basis for $\mathcal{V}_{0}$ and $\mathcal{V}_{1}$.

Note that the principal angles in \cref{defPA} are sorted in increasing order, and it is often the largest principal angle $\angle_{k}(V_{0},V_{1})$ that is of interest. In particular, it is \redit{well known~\cite[Corollary II.4.6]{Stewart1990}} that
$\sin \angle_{k}(V_{0},V_{1}) = \norm{V_{0}V_{0}^{\Ttran}-V_{1}V_{1}^{\Ttran}}$
holds when $k_{0}=k_{1}$, implying that the orthogonal projectors are close for small $\angle_{k}(V_{0},V_{1})$.  
The other principal angles give partial information; knowing that some of them are small allows us to conclude that $\mathcal{V}_{0}$ and $\mathcal{V}_{1}$ nearly share a common subspace. In the context of null spaces, the following result establishes a residual bound in terms of principal angles.
\begin{proposition}
    \label{corWt}
    Consider the setting of \cref{defPA} and a  matrix $A\in\R^{m \times n}$. Then for every $t = 1,\dotsc, k$ there exists a matrix $W_{t}\in\R^{k_{1}\times t}$ with orthonormal columns such that 
    \begin{equation*}
        \norm{AV_{1}W_{t}} \leq \norm{AV_{0}}+\norm{A}\sin\angle_{t}(V_{0},V_{1}).
    \end{equation*}
\end{proposition}
\begin{proof}
    \edit{Let $W_{t}$ contain the right singular vectors corresponding to the smallest $t$ singular values of $(I-V_{0}V_{0}^{\Ttran})V_{1}$. By \cite[Chap.~1, Thm.~4.37]{Stewart1998}, we know that 
    \begin{equation*}
        \sin\angle_{t}(V_{0},V_{1}) = \norm{(I-V_{0}V_{0}^{\Ttran})V_{1}W_{t}}.
    \end{equation*}
    }
    The proof is concluded by the triangular inequality:
    \begin{equation*}
        \begin{aligned}
            \norm{AV_{1}W_{t}} &\leq \norm{AV_{0}V_{0}^{\Ttran}V_{1}W_{t}}+\norm{A(I-V_{0}V_{0}^{\Ttran})V_{1}W_{t}}\\ 
            &\leq \norm{AV_{0}}+\norm{A}\sin\angle_{t}(V_{0},V_{1}).
        \end{aligned}
    \end{equation*}
\end{proof}
We conclude this section with a useful characterization of the tangents of principal angles from \cite[Theorem~3.1]{Zhu2013}.

\begin{proposition}
    \label{proptan} 
    Let $[V_{0},V_{\perp}]$ be an orthonormal matrix, where $V_{0}\in\R^{n\times k_{0}}$. Suppose $\Omega_{1}\in\R^{n\times k_{1}}$ has column full rank. If $\rk(V_{0}^{\Ttran}\Omega_{1})=\rk(\Omega_{1})\leq k_{0}$, then 
    \begin{equation*}    
    \norm{V_{\perp}^{\Ttran}\Omega_{1}(V_{0}^{\Ttran}\Omega_{1})^{\dagger}} = \tan\angle_{k_{1}}(V_{0},\Omega_{1}).
    \end{equation*}
\end{proposition}

\section{Convergence analysis}
\label{Sec:ConvAnal}

In this section, we study the convergence of \cref{algo}, focusing on the case of a single random starting vector $\omega$, that is, $d = 1$.

\subsection{Convergence of single-vector Krylov subspaces for small eigenvalues}

In the following, we analyze how well the Krylov subspace generated in line~3 of \cref{algo} contains approximations to (portions of) the invariant subspace associated with the $N$ smallest eigenvalues of $B$. The analysis mimics analogous results obtained in~\cite{Meyer2023} for low-rank approximation. Compared to classical results~\cite[Chapter VI]{Saad2011}, our analysis differs by considering the convergence to a whole subspace and leveraging the randomness of the initial vector. 
In the following theorem, the main result of this section; $\cheb_{k}(\lambda)$ denotes the Chebyshev polynomial of degree $k$ defined as 
\begin{equation}
    \label{defCheb}
    \cheb_{k}(\lambda) = \frac{1}{2}\big(\big(\lambda+\sqrt{\lambda^{2}-1}\big)^{k}+\big(\lambda+\sqrt{\lambda^{2}-1}\big)^{-k}\big).
\end{equation}

\begin{theorem}
    \label{thmLan}
    For a symmetric matrix $B\in\R^{n\times n}$ with eigenvalues $\lambda_{1}<\lambda_{2}<\dotsb<\lambda_{N}<\lambda_{N+1}\leq \dotsb\leq \lambda_{n}$, let $V_{N} \in \R^{n\times N}$ be an orthonormal basis of the invariant subspace associated with $\lambda_1,\dotsc,\lambda_N$. Choose integers $t,\ell$ such that $1\leq t\leq N$ and $t\leq \ell\leq n$.
    Let $K_{\ell}=[\omega,B\omega,\dotsc,B^{\ell-1}\omega]$, where $\omega\in\R^{n}$ is a Gaussian random vector. Then, for any $0<\delta<1$, the inequality
    \begin{equation}
        \label{eqLan}
        \tan\angle_{t}(V_{N},K_{\ell}) \leq 
        \frac{N\sqrt{5\pi (n-N)\log(n/\delta)}}{\sqrt{2}\delta(t-1)!\cdot\cheb_{\ell-t}\bigl(1+2\frac{\lambda_{N+1}-\lambda_{N}}{\lambda_{n}-\lambda_{N+1}}\bigr)}\Bigl(\frac{2}{\gap}\Bigr)^{t-1}
    \end{equation}
    holds with probability at least $1-2\delta$,
    where
    \begin{equation}
        \label{defgap}
        \gap := \min_{1\leq i\leq N-1}( \lambda_{i+1}-\lambda_{i} ) / (\lambda_{n}-\lambda_{1}).
    \end{equation}   
\end{theorem}

\noindent Combining \cref{thmLan} with \cref{thmgap}, we obtain the following result.
\begin{corollary}
    \label{corLan}
    Suppose that the matrix $B$ in \cref{thmLan} takes the form $B=A^{\Ttran}A+\epsilon D$, where $A$ is an $m\times n$ matrix with null space dimension $N$ and $D$ is a random diagonal matrix with the diagonal entries drawn i.i.d. from the uniform distribution on $[0,1]$. Let $\underline{\sigma}$ and $\overline{\sigma}$ be the smallest and largest nonzero singular values of $A$, respectively. 
    If $0<\epsilon< \underline{\sigma}^{2}$, then 
    \begin{equation}
        \label{coreqtan}
        \tan\angle_{t}(V_{N},K_{\ell}) \leq \frac{N\sqrt{(n-N)\log(n/\delta)}\bigl(C(n^{2}\overline{\sigma}^{2})/(\delta\epsilon)\bigr)^{t-1}}
        {\delta\cdot(t-1)!\cdot\cheb_{\ell-t}\bigl(1+\frac{2(\underline{\sigma}^{2}-\epsilon)}{\overline{\sigma}^{2}-\underline{\sigma}^{2}+\epsilon}\bigr)}
    \end{equation}
    holds with probability at least $1-3\delta$, where $C$ is a universal constant.
\end{corollary}

The rest of this section is concerned with the proof of \cref{thmLan}, which follows the ideas from~\cite{Meyer2023} and is divided into two parts. First, in \cref{thmLan1}, we will establish \cref{eqLan} for $\ell=t$, that is,
\begin{equation}
    \label{eqLan1}
    \tan\angle_{t}(V_{N},K_{t}) \leq \frac{N\sqrt{5\pi (n-N)\log(n/\delta)}}{\sqrt{2}\delta(t-1)!}\Bigl(\frac{2}{\gap}\Bigr)^{t-1}.
\end{equation}
Second, in \cref{thmLan2}, we prove 
\begin{equation}
    \label{eqLan2}
    \tan\angle_{t}(V_{N},K_{\ell})\leq \frac{\tan\angle_{t}(V_{N},K_{t})}{\cheb_{\ell-t}\big(1+2\frac{\lambda_{N+1}-\lambda_{N}}{\lambda_{n}-\lambda_{N+1}}\big)}.
\end{equation}
Plugging \cref{eqLan1} into \cref{eqLan2} then proves the inequality~\cref{eqLan} of \cref{thmLan}.

\subsubsection{The case $\ell = t$. Proof of~\cref{eqLan1}}

\label{thmLan1}
Following~\cite[Theorem~3]{Meyer2023}, the ability of a (single-vector) Krylov subspace to capture a subspace is closely related to properties of Vandermonde matrices.
\begin{lemma}
    \label{lemVan}
    For $\lambda_1 < \dotsb < \lambda_N \le \lambda_{N+1} \le \dotsb \le \lambda_n$ and     $1\leq t\leq N$, consider the 
    Vandermonde matrices 
    \begin{equation}
        \label{defVan}
        \van_{N} = \begin{bmatrix}
            1 & \lambda_{1} & \ldots & \lambda_{1}^{t-1}\\ 
            1 & \lambda_{2} & \ldots & \lambda_{2}^{t-1}\\ 
            \vdots & \vdots & \ddots & \vdots\\
            1 & \lambda_{N} & \ldots & \lambda_{N}^{t-1} 
        \end{bmatrix}, \quad 
        \van_{\perp} = \begin{bmatrix}
            1 & \lambda_{N+1} & \ldots & \lambda_{N+1}^{t-1}\\ 
            1 & \lambda_{N+2} & \ldots & \lambda_{N+2}^{t-1}\\ 
            \vdots & \vdots & \ddots & \vdots\\
            1 & \lambda_{n} & \ldots & \lambda_{n}^{t-1} 
        \end{bmatrix}.
    \end{equation} 
    Then 
    \begin{equation*}
        \norm{\van_{\perp}\van_{N}^{\dagger}}\leq \frac{\sqrt{n-N}}{(t-1)!}\Bigl(\frac{2}{\gap}\Bigr)^{t-1}
    \end{equation*}
    with $\gap$ defined as in~\cref{defgap}.
\end{lemma}

\begin{proof}
For an arbitrary vector $x = [x_{1},\dotsc,x_{t}]^{\Ttran}\in\R^{t}$, we have
    \begin{equation*}
        \norm{\van_{N}x}^{2} = \sum_{i=1}^{N}p^{2}(\lambda_{i})
        \quad \text{and} \quad 
        \norm{\van_{\perp}x}^{2} = \sum_{j=N+1}^{n}p^{2}(\lambda_{j}),
    \end{equation*}
    with the polynomial $p(\lambda) = \sum_{k=1}^{t}x_{k}\lambda^{k-1}$. By the assumptions, 
$\van_{N}$ has full column rank and
    \begin{equation}
        \label{estvan1}
        \norm{\van_{\perp}\van_{N}^{\dagger}}^{2} = \max_{x\in \R^{t}}\frac{\norm{\van_{\perp}x}^{2}}{\norm{\van_{N}x}^{2}} = \max_{p\in \mathcal{P}_{t-1}}\frac{\sum_{j=N+1}^{n}p^{2}(\lambda_{j})}{\sum_{i=1}^{N}p^{2}(\lambda_{i})},
    \end{equation} 
    where $\mathcal{P}_{t-1}$ denotes the vector space containing all polynomials of degree at most $t-1$. Choosing the Lagrange basis 
    \[
        p_{i}(\lambda) = \prod_{\substack{1\leq k\leq t\\ k\neq i}}\frac{\lambda-\lambda_{k}}{\lambda_{i}-\lambda_{k}}\in\mathcal{P}_{t-1}, \quad i = 1, \dotsc, t,
    \] 
    we can express an arbitrary polynomial $p\in\mathcal{P}_{t-1}$ as
    $
        p(\lambda) = p(\lambda_{1})p_{1}(\lambda) + \dotsb  + p(\lambda_{t})p_{t}(\lambda).
    $
    By the Cauchy--Schwarz inequality,
    \begin{equation} \label{eq:cauchyschwartz}
        p^{2}(\lambda) \leq \sum_{i=1}^{t}p^{2}(\lambda_{i})\sum_{i=1}^{t}p_{i}^{2}(\lambda)
        \leq \sum_{i=1}^{N}p^{2}(\lambda_{i})\Bigl(\sum_{i=1}^{t}\abs{p_{i}(\lambda)}\Bigr)^{2}.
    \end{equation}
    To process the last term, we assume that $\lambda_{N+1}\leq \lambda\leq \lambda_{n}$, which implies 
$\abs{\lambda_{i}-\lambda_{k}}/(\lambda-\lambda_{k})\geq \abs{i-k}\cdot\gap$ for $1\leq i\neq k\leq t$. In turn,
    \begin{equation*}
        \sum_{i=1}^{t}\abs{p_{i}(\lambda)}=\sum_{i=1}^{t}\prod_{\substack{1\leq k\leq t\\ k\neq i}}\frac{\lambda-\lambda_{k}}{\abs{\lambda_{i}-\lambda_{k}}}\leq \frac{1}{\gap^{t-1}}\sum_{i=1}^{t}\prod_{\substack{1\leq k\leq t\\ k\neq i}}\frac{1}{\abs{i-k}} = \Bigl(\frac{2}{\gap}\Bigr)^{t-1}\frac{1}{(t-1)!},
    \end{equation*}
    where the last equality utilizes the combinatorial identity
    \begin{equation*}
        \sum_{i=1}^{t}\prod_{\substack{1\leq k\leq t\\ k\neq i}}\frac{1}{\abs{i-k}}=\sum_{i=1}^{t}\frac{1}{(i-1)!(t-i)!}=\frac{1}{(t-1)!}\sum_{i=1}^{t} \binom{t-1}{i-1} =\frac{2^{t-1}}{(t-1)!}.
    \end{equation*}
    Plugging this bound into~\eqref{eq:cauchyschwartz} for $\lambda = \lambda_j$ gives
    \begin{equation*}
        \begin{aligned}
            \frac{\sum_{j=N+1}^{n}p^{2}(\lambda_{j})}{\sum_{i=1}^{N}p^{2}(\lambda_{i})}\leq \frac{n-N}{\bigl((t-1)!\bigr)^{2}}\Bigl(\frac{2}{\gap}\Bigr)^{2(t-1)}.
        \end{aligned}
    \end{equation*}
    Combined with~\cref{estvan1}, this completes the proof.
\end{proof}

By leveraging the randomness of $\omega$, \cref{proptan} allows us to prove~\cref{eqLan1}.
\begin{proof}[Proof of \cref{eqLan1}]
Because $V_{N}^{\Ttran}\omega$ is a Gaussian random vector, the matrix 
$\diagM{V_{N}^{\Ttran}\omega}$ is invertible (almost surely). Using $V_{N}^{\Ttran}K_{t} = \diagM{V_{N}^{\Ttran}\omega}\van_{N}$, the invertibility of this matrix allows us to conclude that the Krylov basis $K_{t} = [\omega,B\omega,\dotsc,B^{t-1}\omega]$ has full rank:
    \begin{equation}
        \label{fullrk}
        \rk(K_{t})\geq \rk(V_{N}^{\Ttran}K_{t}) = \rk(\van_{N})=t,
    \end{equation}
    where the last equality utilizes that the $N$ smallest eigenvalues of $B$ are mutually different.
Letting $V_{\perp}$ denote an orthonormal basis of the invariant subspace associated with $\lambda_{N+1},\dotsc,\lambda_{n}$, we now apply \cref{proptan} to estimate
    \begin{equation}
        \label{decoupleV}
        \begin{aligned}
            \tan^{2}\angle_{t}(V_{N},K_{t}) &=  \norm{V_{\perp}^{\Ttran}K_{t}(V_{N}^{\Ttran}K_{t})^{\dagger}}^{2} \\ 
            &= \norm{\diagM{V_{\perp}^{\Ttran}\omega}\van_{\perp}(\diagM{V_{N}^{\Ttran}\omega}\van_{N})^{\dagger}}^{2}\\
            &\leq \frac{\sigma_{\max}^{2}(\diagM{V_{\perp}^{\Ttran}\omega})}{\sigma_{\min}^{2}(\diagM{V_{N}^{\Ttran}\omega})}\norm{\van_{\perp}\van_{N}^{\dagger}}^{2}.
        \end{aligned}
    \end{equation}
    Because $V_{\perp}^{\Ttran}\omega$ and $V_{N}^{\Ttran}\omega$ are independent Gaussian random vectors, we can use concentration and anti-concentration results for chi-squared random variables from~\cite[Lemma~1]{Laurent2000} and~\cite[Lemma~4]{Meyer2023}, respectively, to conclude that
    \begin{equation*}
        \frac{\sigma_{\max}^{2}(\diagM{V_{\perp}^{\Ttran}\omega})}{\sigma_{\min}^{2}(\diagM{V_{N}^{\Ttran}\omega})}\leq \frac{5\pi N^{2}\log\bigl(n/\delta\bigr)}{2\delta^{2}}
    \end{equation*}
    holds with probability at least $1-2\delta$. Combined with the bound
    on $\norm{\van_{\perp}\van_{N}^{\dagger}}$ from~\cref{lemVan}, this completes the proof of~\cref{eqLan1}.
\end{proof}

\subsubsection{The case $\ell > t$. Proof of~\cref{eqLan2}}
\label{thmLan2}
Our proof crucially relies on a simple but important observation from~\cite[(1.3)]{Meyer2023}, which establishes the following equivalence between the single-vector Krylov subspace \edit{$\mathcal{K}_{\ell}(B,\omega)$} and a block Krylov subspace generated from a specific starting matrix:
\edit{
\begin{equation}
    \label{connectionLanczos}
    \mathcal{K}_{\ell}(B,\omega) = \range[\omega,B\omega,\dotsc,B^{\ell-1}\omega] = \range[K_{t},BK_{t},\dotsc,B^{\ell-t}K_{t}],
\end{equation}}
with $t\leq \ell\leq n$.  The following lemma is a straightforward extension of known convergence results for Krylov subspaces to block Krylov subspaces.

\begin{lemma}
    \label{lemblk}
    For a symmetric matrix $B\in\R^{n\times n}$ with eigenvalues $\lambda_{1}<\lambda_{2}<\dotsb<\lambda_{N}<\lambda_{N+1}\leq \dotsb\leq \lambda_{n}$, let $V_{N} \in \R^{n\times N}$ be an orthonormal basis of the invariant subspace associated with $\lambda_1,\dotsc,\lambda_N$. Choose integers $t,\ell$ such that $1\leq t\leq N$ and $t\leq \ell\leq n$.
    For an $n\times t$ matrix $\Omega$ such that $\rk(V_{N}^{\Ttran}\Omega)=t$, let $K=[\Omega,B\Omega,\dotsc,B^{\ell-t}\Omega]\in\R^{n\times(\ell - t + 1)t}$. Then
    \begin{equation} \label{eq:lemblk}
        \tan\angle_{t}(V_{N},K)\leq \frac{\tan\angle_{t}(V_{N},\Omega)}{\cheb_{\ell-t}\Bigl(1+2 \frac{\lambda_{N+1}-\lambda_{N}}{\lambda_{n}-\lambda_{N+1}}\Bigr)},
    \end{equation}
    where $\cheb_{\ell-t}$ is the Chebyshev polynomial~\cref{defCheb} of degree $\ell - t$.
\end{lemma}    
\begin{proof}
Let $p(\lambda) = \alpha_0 + \alpha_1 \lambda + \dotsb \alpha_{\ell -t} \lambda^{\ell - t}$ be such that $p(\lambda) \not=0$ for 
$\lambda_{1}\leq \lambda\leq \lambda_{N}$. Defining the matrix $H_{K} = [\alpha_{0}I_{t},\alpha_{1}I_{t},\dotsc,\alpha_{\ell-t}I_{t}]^{\Ttran}$, we have that 
    \begin{equation*}
        V_{N}^{\Ttran}KH_{K} = \sum_{i=0}^{\ell-t}\alpha_{i}V_{N}^{\Ttran}B^{i}\Omega = \sum_{i=0}^{\ell-t}\alpha_{i}\Lambda_{N}^{i}V_{N}^{\Ttran}\Omega=p(\Lambda_{N})V_{N}^{\Ttran}\Omega,
    \end{equation*}
    and an analogous formula holds when $V_{N}$ is replaced by $V_\perp$.
    By assumption, $p(\Lambda_{N})$ is invertible and thus $V_{N}^{\Ttran}KH_{K}$ has rank $t$. Additionally, using that
    \edit{$\range(KH_{K})\subset \range (K)$}
    and 
    the expression for $\tan\angle_{t}$ from \cref{proptan}, we obtain that
    \redit{
    \begin{equation*}
        \begin{aligned}
            \tan\angle_{t}(V_{N},K)&\leq \tan\angle_{t}(V_{N},KH_{K})= \norm{V_{\perp}^{\Ttran}KH_{K}(V_{N}^{\Ttran}KH_{K})^{\dagger}} \\ 
            &=\max_{x\in\R^{t}}\frac{\norm{V_{\perp}^{\Ttran}KH_{K}x}}{\norm{V_{N}^{\Ttran}KH_{K}x}}
            =\max_{x\in\R^{t}}
            \frac{\norm{p(\Lambda_{\perp})V_{\perp}^{\Ttran}\Omega x}}{\norm{p(\Lambda_{N})V_{N}^{\Ttran}\Omega x}}\\
            &\leq \frac{\max\limits_{N+1\leq j\leq n}\abs{p(\lambda_{j})}}{\min\limits_{1\leq i\leq N}\abs{p(\lambda_{i})}}
            \max_{x\in\R^{t}}\frac{\norm{V_{\perp}^{\Ttran}\Omega x}}{\norm{V_{N}^{\Ttran}\Omega x}} 
            = \frac{\max\limits_{N+1\leq j\leq n}\abs{p(\lambda_{j})}}{\min\limits_{1\leq i\leq N}\abs{p(\lambda_{i})}}\norm{V_{\perp}^{\Ttran}\Omega(V_{N}^{\Ttran}\Omega)^{\dagger}}\\ 
            &=\frac{\max\limits_{N+1\leq j\leq n}\abs{p(\lambda_{j})}}{\min\limits_{1\leq i\leq N}\abs{p(\lambda_{i})}}\tan\angle_{t}(V_{N},\Omega).
        \end{aligned}
    \end{equation*}}
    The proof is completed by setting
    $p(\lambda) = \cheb_{\ell-t}\bigl(1+\frac{2(\lambda_{N+1}-\lambda)}{\lambda_{n}-\lambda_{N+1}}\bigr)$ and using that a Chebyshev polynomial is bounded by one in magnitude inside $[-1,1]$ and monotone outside that interval.
\end{proof}

\subsection{Convergence of \cref{algo} \edit{with a single initial vector}}

In this subsection, we investigate the quality of the approximate null space produced by \cref{algo} via applying the Rayleigh--Ritz procedure to the Krylov subspace.
Unlike existing results on the block Lanczos method~\cite{Li2015}, we also present the convergence of lower-dimensional subspaces by bounding the principal angles $\angle_{t}$ for all $1 \leq t \leq N$. We start with a result that establishes the quasi-optimality of Rayleigh--Ritz.

\begin{lemma}
    \label{lemAVW}
    Under the setting of~\cref{thmVANg}, let $V$ be the approximate null space basis returned after $\ell\geq N$ iterations of \cref{algo}.
    Let $\tilde{\lambda}_{t}^{(0)}$ be the $t$th smallest eigenvalue of $Z_{\ell}^{\Ttran}A^{\Ttran}AZ_{\ell}$, where $1\leq t\leq \ell$ and $Z_{\ell}$ is an orthonormal basis of \edit{$\mathcal{K}_{\ell}(B,\omega)$}.
    Then for all $1\leq t\leq N$, there exists a matrix $W_{t}\in\R^{N\times t}$ with orthonormal columns such that 
    \begin{equation}
        \label{estAVWt}  
        \norm{AVW_{t}}^{2} \leq \tilde{\lambda}_{t}^{(0)} +\epsilon.
    \end{equation}

    Furthermore, if $t=N$ and $0<\epsilon<\underline{\sigma}^{2}-\tilde \lambda_N^{(0)}$, then 
    \begin{equation}
        \label{estAVWN}
        \norm{AV}^{2} \leq \tilde{\lambda}_{N}^{(0)} +\frac{\epsilon^{2}}{\underline{\sigma}^{2}-\tilde{\lambda}_{N}^{(0)} - \epsilon}.
    \end{equation}
\end{lemma}
\begin{proof}
    Let $V_{Z}$ and $V_{Z}W_{t}$ be two matrices with orthonormal columns, where their columns form orthonormal bases for the invariant subspaces associated with the $N$ and $t$ smallest eigenvalues of $Z_{\ell}^{\Ttran}BZ_{\ell}$, respectively. By definition, $V=Z_{\ell}V_{Z}$.
The gap-independent perturbation result in \cref{propEvalPert} allows us to conclude~\eqref{estAVWt}:
\begin{equation*}
    \norm{AVW_{t}}^{2}=\norm{(V_{Z}W_{t})^{\Ttran}( Z_\ell^{\Ttran} A^{\Ttran} A Z_\ell ) (V_{Z}W_{t}) } \le \tilde \lambda_t^{(0)} + \norm{Z_\ell^{\Ttran} (B-A^{\Ttran} A) Z_\ell} \le 
 \tilde \lambda_t^{(0)} + \epsilon,
\end{equation*}
where we used that $\tilde \lambda_t^{(0)}$ is the $t$th smallest eigenvalue of $Z_\ell^{\Ttran} A^{\Ttran} A Z_\ell$.

For \cref{estAVWN}, the result when $\ell=N$ is straightforward since $\norm{AV}^{2}=\tilde{\lambda}_{N}^{(0)}$. To show the case $\ell>N$,
we first note that eigenvalue interlacing implies
$\tilde \lambda_{N+1}^{(0)}\geq \underline{\sigma}^{2}$.
Because of $\| Z_\ell^{\Ttran} B Z_\ell -  Z_\ell^{\Ttran} A^{\Ttran} A Z_\ell\| \le \epsilon < \underline{\sigma}^{2}-\tilde \lambda_N^{(0)} \le \tilde \lambda_{N+1}^{(0)} - \tilde \lambda_{N}^{(0)}$, 
we can apply the second, gap-dependent result in \cref{propEvalPert} to establish~\eqref{estAVWN}.
\end{proof}

By combining the eigenvalue repulsion result from \cref{thmgap}, the Krylov subspace convergence result from \cref{thmLan} and the quasi-optimality of Rayleigh--Ritz from \cref{lemAVW},
we can conclude our main convergence result on the convergence of \cref{algo}.

\begin{theorem}
    \label{thmmain}
    Let $\overline{\sigma}$ and $\underline{\sigma}$ denote the largest and smallest nonzero singular values of a matrix $A\in\R^{m \times n}$ having null space dimension $N$. Let $V$ be the approximate null space basis returned by \cref{algo} with $0<\epsilon<\underline{\sigma}^{2}/3$ and $N\leq \ell\leq n$. For every $1\leq t\leq N$, there exists a matrix $W_{t}\in\R^{N\times t}$ with orthonormal columns such that 
    \begin{equation}
        \label{estmaint}
        \norm{AVW_{t}}^{2}\leq \Bigl(\frac{\epsilon}{(\underline{\sigma}^{2}-\epsilon)^{1/2}}+\overline{\sigma}\tan\angle_{t}(V_{N},Z_{\ell})\Bigr)^{2}+\epsilon,
    \end{equation}
    holds with probability at least $1-3\delta$,
    where $\tan\angle_{t}(V_{N},Z_{\ell})$ is bounded according to~\cref{coreqtan}. For $t=N$, when
    \begin{equation}
        \label{contan}
        \tan\angle_{N}(V_{N},Z_{\ell})\leq \frac{2\underline{\sigma}^{2}-5\epsilon}{3\overline{\sigma}(\underline{\sigma}^{2}-\epsilon)^{1/2}},
    \end{equation}
    the bound can be improved to 
    \begin{equation}
        \label{estmainN}
        \norm{AV}^{2}\leq \Bigl(\frac{\epsilon}{(\underline{\sigma}^{2}-\epsilon)^{1/2}}+\overline{\sigma}\tan\angle_{N}(V_{N},Z_{\ell})\Bigr)^{2}+\frac{9\epsilon^{2}}{5(\underline{\sigma}^{2}-\epsilon)}.
    \end{equation}
\end{theorem}

\edit{The convergence results in \cref{estmaint,estmainN} consist of two components. The first component (terms in parentheses) captures the convergence of the Krylov subspace method to produce approximations to eigenvalues and eigenvectors of $B$. 
The second component, discussed in \cref{lemAVW}, quantifies the quasi-optimality of the Rayleigh--Ritz procedure when applied to $B$ instead of $A^{\Ttran}A$. 
When measuring the convergence for all zero singular values in \cref{estmainN}, the magnitudes of these two components are both of order $\epsilon^{2}$,} \editr{and this result is sharp up to a constant for some matrices.}

\begin{proof}[Proof of \cref{thmmain}]
    Using~\cref{lemAVW}, it suffices to bound $\tilde{\lambda}_t^{(0)}$, the $t$th smallest eigenvalue of $Z_{\ell}^{\Ttran}A^{\Ttran}AZ_{\ell}$. Using~\cref{corWt}, 
    \cref{thmVANg}, and $\abs{\sin\angle}\leq \abs{\tan\angle}$, we derive 
\[
    \sqrt{\tilde \lambda_t^{(0)}} \leq \norm{AV_{N}}+\norm{A}\sin\angle_{t}(V_{N},Z_{\ell})\leq \frac{\epsilon}{(\underline{\sigma}^{2}-\epsilon)^{1/2}}+\overline{\sigma}\tan\angle_{t}(V_{N},Z_{\ell}).  
\]
Combined with~\cref{estAVWt}, this gives~\cref{estmaint}.

In order to prove \cref{estmainN}, we first note that the condition~\cref{contan} implies
\begin{equation*}
    \underline{\sigma}^{2}-\epsilon-\tilde \lambda_N^{(0)}
    \geq \underline{\sigma}^{2}-\epsilon-\Bigl(\frac{\epsilon}{(\underline{\sigma}^{2}-\epsilon)^{1/2}}+\overline{\sigma}\tan\angle_{N}(V_{N},Z_{\ell})\Bigr)^{2}\geq \frac{5(\underline{\sigma}^{2}-\epsilon)}{9}>0.
\end{equation*}
Thus, the condition for the bound~\cref{estAVWN} in \cref{lemAVW} is met, which implies the desired result \cref{estmainN}.
\end{proof}

\begin{remark}
    \label{cormain}
    Let $\kappa=\overline{\sigma}/\underline{\sigma}$ be the effective condition number of $A$, and $\epsilon=c\underline{\sigma}^{2}$, where $0<c<1/3$.
    \cref{thmmain} states that, with probability at least $1-3\delta$,
    \begin{equation*}
        \ell = \order\biggl(1+\frac{\log\bigl((n\kappa)/(t\delta\epsilon)\bigr)}{\log(1+2\kappa^{-1})}\biggr)t
    \end{equation*}
    iterations of \cref{algo} are needed to partially approximate the null space of $A$ in the sense of
    \begin{equation*}
        \tan\angle_{t}(V_{N},Z_{\ell}) \leq \frac{\epsilon}{\sqrt{6}\overline{\sigma}\underline{\sigma}}\leq \frac{2\underline{\sigma}^{2}-5\epsilon}{3\overline{\sigma}(\underline{\sigma}^{2}-\epsilon)^{1/2}},
        \quad\text{and}\quad 
        \norm{AVW_{t}}^{2}\leq \frac{17\epsilon}{9},
    \end{equation*}
    where we used Stirling's formula and a known bound for Chebyshev polynomials in \cite[Section~4.4.1]{Saad2011}. 
    Let $BZ_{\ell}=Z_{\ell}T_{\ell}+Q_{\ell+1}E_{\ell+1}$ represent the Lanczos decomposition computed by \cref{algo}. Then, $T_{\ell}-3\epsilon I$ has at least $t$ negative eigenvalues, all originating from Ritz vectors close to the null space. This assertion stems from the fact that the $(N + 1)$st smallest eigenvalue of $B$ is at least $\underline{\sigma}^{2}=\epsilon/c>3\epsilon$. Moreover, for $t=N$, the error bound improves to 
    \begin{equation*}
        \norm{AV}^{2}\leq \Bigl(\frac{\epsilon}{(\underline{\sigma}^{2}-\epsilon)^{1/2}}+\frac{\epsilon}{\sqrt{6}\underline{\sigma}}\Bigr)^{2}+\frac{9\epsilon^{2}}{5(\underline{\sigma}^{2}-\epsilon)} \leq \frac{161\epsilon^{2}}{30\underline{\sigma}^{2}}.
    \end{equation*}
\end{remark}

Our theoretical results highlight a crucial distinction between the convergence behaviors of the single-vector and block Lanczos methods. In contrast to the block Lanczos method~\cite{Li2015}, where Ritz vectors converge collectively, the single-vector Lanczos method exhibits individual convergence of Ritz vectors. 
This individual convergence allows for deflation and adaptive detection of the null space dimension, which is particularly advantageous in the typical scenario where prior information about $N$ is unavailable.

\section{Practical techniques for implementation}
\label{sec:implementation}
In this section, we discuss several improvements and implementation details to make \cref{algo} competitive.
\subsection{Restarting}
\cref{algo} needs to store the orthonormal basis $Z_{\ell}$ of the Krylov subspace to allow for reorthogonalization (see below) and Ritz vector computations. As the number of iterations $\ell$ required for convergence increases, storing $Z_{\ell}$ becomes challenging due to memory limitations. Additionally, the cost of Rayleigh--Ritz extraction grows significantly with $\ell$. To address these issues, we employ a restart mechanism, akin to the Krylov--Schur method \cite{Stewart2002, Zhou2008,Wu2000}.

To conveniently describe our restarting mechanism, we use the superscripts $\cdot^{(\old)}$ and $\cdot^{(\new)}$ to denote variables before and after restarting, respectively.
We assume that the dimension of the Krylov subspace before restarting is $\ell d$ and that 
$\ell_{0}d$ Ritz pairs are retained after restarting, where $\ell_0 < \ell$ and $d$ represents the block size of the initial matrix $\Omega$. 

Suppose we have a Lanczos decomposition of the form
\begin{equation}
    \label{laniter}
    BZ_{\ell}^{(\old)} = Z_{\ell}^{(\old)}T_{\ell}^{(\old)}+Q_{\ell+1}^{(\old)}E_{\ell+1}^{(\old)},
\end{equation} 
where $[Z_{\ell}^{(\old)},Q_{\ell+1}^{(\old)}]$ is a matrix with orthonormal columns, $E_{\ell+1}^{(\old)}=[0,\Gamma_{\ell}^{(\old)}]$, and 
\edit{
\begin{equation}
    \label{defT}
    T_{\ell}^{(\old)} = \begin{bmatrix}
        \Theta^{(\old)} & (F^{(\old)})^{\Ttran} &&&&\\
        F^{(\old)} &\Ta_{\ell_{0}+1}^{(\old)} & (\Tb_{\ell_{0}+1}^{(\old)})^{\Ttran} & & &\\ 
        &\Tb_{\ell_{0}+1}^{(\old)} & \Ta_{\ell_{0}+2}^{(\old)} & (\Tb_{\ell_{0}+2}^{(\old)})^{\Ttran} & &\\ 
        && \ddots & \ddots & \ddots &\\ 
        &&&\Tb_{\ell-2}^{(\old)} & \Ta_{\ell-1}^{(\old)} & (\Tb_{\ell-1}^{(\old)})^{\Ttran}\\
        &&&&\Tb_{\ell-1}^{(\old)} & \Ta_{\ell}^{(\old)} \\
    \end{bmatrix}
\end{equation}}
is a block tridiagonal matrix, partitioned such that $\Theta^{(\old)} \in \R^{\ell_0 d \times \ell_0 d}$ contains the information from the previous cycle. To restart and keep $\ell_{0}d$ Ritz vectors, we apply the Rayleigh--Ritz procedure:
\begin{equation*}
    T_{\ell}^{(\old)}V_{Z}^{(\new)}=V_{Z}^{(\new)}\Theta^{(\new)} 
    \quad \text{and} \quad 
    Z_{\ell_{0}}^{(\new)} = Z_{\ell}^{(\old)}V_{Z}^{(\new)} \in \R^{n \times \ell_0 d},
\end{equation*}
where $V_{Z}^{(\new)} \in \R^{\ell d \times \ell_0 d}$ is an orthonormal basis of the invariant subspace corresponding to the $\ell_{0}d$ smallest eigenvalues of $T_{\ell}^{(\old)}$. Multiplying $V_{Z}^{(\new)}$ from the right to the Lanczos decomposition \cref{laniter} leads to 
\begin{equation*}
    BZ_{\ell_{0}}^{(\new)} = Z_{\ell_{0}}^{(\new)}\Theta^{(\new)}+Q_{\ell+1}^{(\old)}E_{\ell+1}^{(\old)}V_{Z}^{(\new)}.
\end{equation*} 
We restart the block Lanczos method by defining the quantities
\begin{equation*}
    \begin{aligned}
        &Q_{\ell_{0}+1}^{(\new)} = Q_{\ell+1}^{(\old)}
          \quad \text{and}\quad
        Z_{\ell_{0}+1}^{(\new)} = [Z_{\ell_{0}}^{(\new)},Q_{\ell_{0}+1}^{(\new)}] \in \R^{n\times (\ell_0+1) d},\\
        & \edit{F^{(\new)} = E_{\ell+1}^{(\old)}V_{Z}^{(\new)}}
        \quad\text{and}\quad 
        \Ta_{\ell_{0}+1}^{(\new)} = (Q_{\ell_{0}+1}^{(\new)})^{\Ttran}BQ_{\ell_{0}+1}^{(\new)},\\ 
        &\bigl(I-Z_{\ell_{0}+1}^{(\new)}(Z_{\ell_{0}+1}^{(\new)})^{\Ttran}\bigr)BQ_{\ell_{0}+1}^{(\new)} = Q_{\ell_{0}+2}^{(\new)}\Tb_{\ell_{0}+1}^{(\new)},
    \end{aligned}
\end{equation*}
where the last line denotes a QR decomposition. After restarting, the new Lanczos decomposition takes the form
\begin{equation*}
    \begin{aligned}
        BZ_{\ell_{0}+1}^{(\new)} &= [BZ_{\ell_{0}}^{(\new)},BQ_{\ell_{0}+1}^{(\new)}] \\ 
        &= [Z_{\ell_{0}}^{(\new)},Q_{\ell_{0}+1}^{(\new)}]\begin{bmatrix}
            \Theta^{(\new)} & (E_{\ell+1}^{(\old)}V_{Z}^{(\new)})^{\Ttran}\\E_{\ell+1}^{(\old)}V_{Z}^{(\new)} & \Ta_{\ell_{0}+1}^{(\new)}
        \end{bmatrix}+Q_{\ell_{0}+2}^{(\new)}[O,\Gamma_{\ell_{0}+1}^{(\new)}]\\ 
        &=Z_{\ell_{0}+1}^{(\new)}\edit{\begin{bmatrix}
            \Theta^{(\new)} & (F^{(\new)})^{\Ttran}\\F^{(\new)} & \Ta_{\ell_{0}+1}^{(\new)}
        \end{bmatrix}}+Q_{\ell_{0}+2}^{(\new)}E_{\ell_{0}+2}^{(\new)},
    \end{aligned}
\end{equation*}
where $E_{\ell_{0}+2}^{(\new)}=[O,\Gamma_{\ell_{0}+1}^{(\new)}]$, and the relationship
\begin{equation*}
    E_{\ell+1}^{(\old)}V_{Z}^{(\new)} = (Q_{\ell+1}^{(\old)})^{\Ttran}(BZ_{\ell_{0}}^{(\new)}-Z_{\ell_{0}}^{(\new)}\Theta^{(\new)})= (Q_{\ell_{0}+1}^{(\new)})^{\Ttran}BZ_{\ell_{0}}^{(\new)}
\end{equation*}
is used.

\subsection{Partial reorthogonalization}
\label{sec:RO}
In floating-point arithmetic, round-off error compromises the orthogonality of the basis computed by the block Lanczos me\-thod, leading to the emergence of multiple converged Ritz values, so called ghost eigenvalues. To obtain reliable information about the null space and its dimension, it is imperative to maintain the orthogonality of Lanczos vectors with reorthogonalization strategies.

Commonly employed reorthogonalization schemes include selective reorthogonalization \cite{Parlett1979}, partial reorthogonalization \cite{Simon1984}, and full reorthogonalization. In our implementation, we use partial reorthogonalization, which is cheaper than full reorthogonalization and more effective than selective reorthogonalization~\cite{Simon1982}. Following the strategies described in \cite{Wu2000,Grimes1994}, we define a $W$-recurrence to monitor loss of orthogonality. For brevity, we omit the superscripts $\cdot^{(\old)}$ and $\cdot^{(\new)}$ regarding restarting. Additionally, we focus on the reorthogonalization of the new basis vector $Q_{\ell+1}$ versus $Z_{\ell}=[Q_{1},\dotsc,Q_{\ell}]$.

Supposing that $\ell_{0}d$ Ritz vectors were retained in the last restart, and
\begin{equation*}
    \Theta=\diagm \{\Theta_{1},\dotsc,\Theta_{\ell_{0}}\}
    \quad\text{and}\quad 
    \edit{F=[F_{1},\dotsc,F_{\ell_{0}}]},
\end{equation*} 
the recurrence~\cref{laniter} in finite precision can be written as 
\begin{equation}
    \label{laniterfp}
    \begin{aligned}
        BQ_{i} &= Q_{i}\Theta_{i}+Q_{\ell_{0}}\edit{F_{i}}+\Delta_{i}\quad \text{for}\quad 1\leq i\leq \ell_{0},\\ 
        BQ_{\ell_{0}+1} &= \sum_{i=1}^{\ell_{0}}Q_{i}\edit{F_{i}^{\Ttran}}+Q_{\ell_{0}+1}\Ta_{\ell_{0}+1}+Q_{\ell_{0}+2}\edit{\Tb_{\ell_{0}+1}}+\Delta_{\ell_{0}+1},\\ 
        BQ_{j} &= Q_{j-1}\edit{\Tb_{j-1}^{\Ttran}}+Q_{j}\Ta_{j}+Q_{j+1}\edit{\Tb_{j}}+\Delta_{j}\quad \text{for}\quad \ell_{0}+2\leq j\leq \ell,        
    \end{aligned}
\end{equation}
where the matrices denoted with $\Delta$ account for round-off errors. Following~\cite{Grimes1994,Parlett1998}, we may assume that the matrices $Q_{k}$ have orthonormal columns and 
\begin{equation}
    \label{defmprd}
    \norm{\Delta_{k}}\leq d\sqrt{n}\mpr\norm{B}=: \mpr_{d}
\end{equation}
for $1\leq k\leq \ell+1$, where $\mpr$ denotes the machine precision.
Different from full orthogonalization, we allow $Q_{i}$ and $Q_{j}$ to be not (numerically) orthogonal to each other for $i\not=j$.
The basic idea of partial reorthogonalization is to monitor this loss of orthogonality by examining $\norm{Q_{k}^{\Ttran}Q_{\ell+1}}$ for $1 \leq k \leq \ell$. If this value exceeds a certain threshold $\mpr_{\mathsf{RO}}$, typically chosen as $\mpr_{\mathsf{RO}}=\sqrt{\mpr}$, we perform reorthogonalization. Because computing $\norm{Q_{k}^{\Ttran}Q_{\ell+1}}$ is computationally expensive, we instead aim to calculate and use an upper bound for this value. Reorthogonalization is not performed if the upper bound is below $\mpr_{\mathsf{RO}}$, potentially cutting down computational costs while maintaining a decent level of orthogonality. 

To derive an upper bound for $\norm{Q_{k}^{\Ttran}Q_{\ell+1}}$, we take $j=\ell$ in the third recurrence in \cref{laniterfp}. Multiplying with $Q_{k}^{\Ttran}$ from the left and rearranging terms yields the inequality
    \begin{align*}
        \norm{Q_{k}^{\Ttran}Q_{\ell+1}}&\leq \norm{\Gamma_{\ell}^{-1}}\bignorm{Q_{k}^{\Ttran}BQ_{\ell}-Q_{k}^{\Ttran}Q_{\ell-1}\edit{\Gamma_{\ell-1}^{\Ttran}}-Q_{k}^{\Ttran}Q_{\ell}\Ta_{\ell}-Q_{k}^{\Ttran}\Delta_{\ell}}\\ 
        &\leq \norm{\Gamma_{\ell}^{-1}}\bigl(\norm{Q_{k}^{\Ttran}BQ_{\ell}}+\norm{Q_{k}^{\Ttran}Q_{\ell-1}}\norm{\Tb_{\ell-1}}+\norm{Q_{k}^{\Ttran}Q_{\ell}}\norm{\Ta_{\ell}}+\norm{Q_{k}}\norm{\Delta_{\ell}}\bigr).
    \end{align*} 
For $1\leq k\leq \ell_{0}$, the term $\norm{Q_{k}^{\Ttran}BQ_{\ell}}$ is bounded by using the symmetry of $B$ and the first recurrence in~\cref{laniterfp}:
\begin{equation*}
    \begin{aligned}
        \norm{Q_{k}^{\Ttran}BQ_{\ell}}&=\norm{Q_{\ell}^{\Ttran}BQ_{k}}\leq \norm{Q_{k}^{\Ttran}Q_{\ell}}\norm{\Theta_{k}}+\norm{Q_{\ell_{0}}^{\Ttran}Q_{\ell}}\norm{F_{k}}+\norm{Q_{\ell}}\norm{\Delta_{k}}. 
    \end{aligned}
\end{equation*} 
Combining these two inequalities and using \cref{defmprd} yields
\begin{equation*}
    \norm{Q_{k}^{\Ttran}Q_{\ell+1}} \leq \norm{\Gamma_{\ell}^{-1}}\bigl(\norm{Q_{k}^{\Ttran}Q_{\ell}}(\norm{\Theta_{k}}+\norm{\Ta_{\ell}})+\norm{Q_{\ell_{0}}^{\Ttran}Q_{\ell}}\norm{F_{k}}+\norm{Q_{k}^{\Ttran}Q_{\ell-1}}\norm{\Tb_{\ell-1}}+2\mpr_{d}\bigr).
\end{equation*}

We keep track of the recursive inequalities above using a strictly upper triangular matrix $W$, such that each entry $w_{i,j}$ is an upper bound on $\norm{Q_{i}^{\Ttran}Q_{j}}$ for $1\leq i< j\leq \ell+1$. Assuming that columns $1,\ldots, \ell$ have been determined, the entries $1,\ldots, \ell$ of column $\ell+1$ are determined by the recurrences above as follows:
\begin{equation*}
    w_{k,\ell+1} = \left\{
    \begin{aligned}
        \norm{\Gamma_{\ell}^{-1}}\Bigl(&w_{\ell_{0}+1,\ell}(\norm{\Theta_{k}}+\norm{\Ta_{\ell}})+w_{\ell_{0},\ell} \norm{F_{k}}\\ 
        &+w_{k,\ell-1}\norm{\Tb_{\ell-1}}+2\mpr_{d}\Bigr), &\!\!\!\!\!\!\!\!\text{for } 1\leq k\leq \ell_{0}.\\ 
        \norm{\Gamma_{\ell}^{-1}}\Bigl(&w_{k,\ell}(\norm{\Ta_{k}}+\norm{\Ta_{\ell}})+\sum_{1\leq i\leq \ell_{0}}w_{i,\ell } \norm{F_{i}}\\ 
        &+w_{k+1,\ell}\norm{\Tb_{k}} + w_{k,\ell-1}\norm{\Tb_{\ell-1}}+2\mpr_{d}\Bigr),&\!\!\!\!\!\!\text{for } k=\ell_{0}+1.\\ 
        \norm{\Gamma_{\ell}^{-1}}\Bigl(&w_{k,\ell}(\norm{\Ta_{k}}+\norm{\Ta_{\ell}})+w_{k+1,\ell}\norm{\Tb_{k}}\\ 
        &+w_{k-1,\ell}\norm{\Tb_{k-1}}+w_{k,\ell-1}\norm{\Tb_{\ell-1}}+2\mpr_{d}\Bigr),&\text{for } \ell_{0}+2\leq k\leq \ell.
    \end{aligned}
    \right.
\end{equation*}
Note that the computation of the column $\ell+1$ only involves the previous two columns $\ell-1$ and $\ell$, consistent with the three-term recurrence of the Lanczos method.
The entries of the initial two columns $\ell_{0}+1$ and $\ell_{0}+2$ of $W$ are set to 
$\mpr_{d}$ because $Q_{\ell_{0}+1}$ and $Q_{\ell_{0}+2}$ are reorthogonalized versus $Z_{\ell_{0}}$ immediately after restarting.
We monitor the loss of orthogonality using $W$.
If $w_{k,\ell+1} >\mpr_{\mathsf{RO}}$ for some $k$, we reorthogonalize $Q_{\ell}$ and $Q_{\ell+1}$ versus $Z_{\ell-1}$ and set the strictly upper triangular entries of the corresponding columns $\ell$ and $\ell+1$ of $W$ to $\mpr_{d}$.

\subsection{Pseudocode}

\cref{algoad} incorporates the improvements to~\cref{algo} discussed in this section. For simplicity, we assume that $\dim\mathcal{K}$ is an integer multiple of $d$. The pseudocode uses Matlab's colon notation to denote submatrices.
\begin{algorithm2e}[htbp]
	\caption{Improved randomized small-block Lanczos} 
    \label{algoad}
	\KwIn{Matrix $A\in\R^{m \times n}$ with $m \ge n$ and an unknown number of zero singular values. Perturbation parameter $\epsilon>0$. Block size $d$. Maximum allowed Krylov subspace dimension $\dim \mathcal{K}$.}
    \KwOut{Orthonormal basis $V$ approximating the null space of $A$.}
    Set $B=A^{\Ttran}A+\epsilon D$, with a diagonal matrix $D$ having diagonal entries uniformly i.i.d. in $[0,1]$\;
    Draw Gaussian random  matrix $\Omega\in\R^{n\times d}$\;
    Set $\mpr_{d}=d\sqrt{n}\mpr\norm{B}$ and $\ell_{0}=0$\;
    Compute $Q_{1}$ as an orthonormal basis of $\Omega$\;
    \While{not converged}{
        Compute $Q_{\ell_{0}+1}^{\prime}=BQ_{\ell_{0}+1}$ and  $\Ta_{\ell_{0}+1} = Q_{\ell_{0}+1}^{\Ttran}Q_{\ell_{0}+1}^{\prime}$\;
        Compute QR decomposition $(I-Z_{\ell_{0}+1}Z_{\ell_{0}+1}^{\Ttran})Q_{\ell_{0}+1}^{\prime}=Q_{\ell_{0}+2}\Tb_{\ell_{0}+1}$\; 
        Set $Z_{\ell_{0}+2}=[Z_{\ell_{0}+1},Q_{\ell_{0}+2}]$\;
        \For{$\ell=\ell_{0}+2,\dotsc,\dim\mathcal{K}/d$}{
            Compute $Q_{\ell}^{\prime} = BQ_{\ell}-Q_{\ell-1}\Tb_{\ell-1}$ and $\Ta_{\ell} = Q_{\ell}^{\Ttran}Q_{\ell}^{\prime}$\;
            Compute QR decomposition $Q_{\ell}^{\prime}-Q_{\ell}\Ta_{\ell} =Q_{\ell+1}\Tb_{\ell}$\;
            Perform partial reorthogonalization as discussed in \cref{sec:RO} and set $Z_{\ell}=[Z_{\ell-1},Q_{\ell}]$ \;
        }
        Form block tridiagonal matrix $T_{\ell}$ defined in \cref{defT} and compute its eigenvalue decomposition $T_{\ell}V_{Z}=V_{Z}\Theta$ with eigenvalues in ascending order\; 
        Let $N_{\ell}$ be the number of eigenvalues of $T_{\ell}$ below $3\epsilon$ and set $\ell_{0}=(N_{\ell}+\dim\mathcal{K})/2$\;
        Set $\Theta = \Theta(1:\ell_{0},1:\ell_{0})$ and compute $Z_{\ell_{0}}=Z_{\ell}V_{Z}(:,1:\ell_{0})$\;
        Compute QR decomposition $(I-Z_{\ell_{0}}Z_{\ell_{0}}^{\Ttran})Q_{\ell+1}=Q_{\ell_{0}+1}\Tb_{\ell_{0}}$\;
        Set \edit{$F=[O,\Gamma_{\ell_{0}}]V_{Z}$} and $Z_{\ell_{0}+1} = [Z_{\ell_{0}},Q_{\ell_{0}+1}]$\; 
    }
    \Return{$V=Z_{\ell}V_{Z}(:,1:N_{\ell})$}\;
\end{algorithm2e}

\edit{Now, let us discuss detecting convergence without prior knowledge of the nullity~$N$. 
In the Lanczos method, Ritz values typically converge first to the exterior eigenvalues.
An intuitive approach is to terminate the algorithm once an unwanted small Ritz value exceeding $3\epsilon$ has converged. However, this approach is slow if $\lambda_{N+2} - \lambda_{N+1}$ is small. Alternatively, we can leverage the individual convergence property of the small-block Lanczos method.
Based on the theoretical results discussed in \cref{cormain} for single-vector Lanczos method, as well as numerical observations for its small-block variants, the convergence speed $s$---measured by the number of matrix-vector multiplications required for $d$ newly converged Ritz values---tends to be stable.
By estimating the convergence speed during computation, we terminate the algorithm if no additional Ritz values have converged after $s$ matrix-vector multiplications.}

\subsection{Preconditioning}

In \cref{cormain}, we established that the number of iterations $\ell$ needed for convergence depends on the effective condition number $\kappa=\overline{\sigma}/\underline{\sigma}$. Preconditioning can be used to reduce this number.

\subsubsection{Inner preconditioning}
Given an $m\times m$ symmetric positive preconditioner $P_{L}$ for $A A^{\Ttran}$, one may replace $B=A^{\Ttran}A+\epsilon D$ by \[ B_{L}=A^{\Ttran}P_{L}^{-1}A+\epsilon D \] within the Lanczos method. Besides maintaining the null space (the null spaces of $A^{\Ttran}P_{L}^{-1}A$ and $A$ are identical), the effective condition number is improved to $\kappa = \norm{(A^{\Ttran}P_{L}^{-1}A)(A^{\Ttran}P_{L}^{-1}A)^{\dagger}}^{1/2}$. For $m\gg n$, using such a preconditioner may be expensive because  $P_{L}$ is an $m \times m$ matrix.

\subsubsection{Outer preconditioning} \label{sec:outerprecond}
Given an invertible matrix  $P_{R}\in\R^{n\times n}$ such that $P_{R}P_{R}^{\Ttran}$ is a preconditioner for $A^{\Ttran} A$, one may replace $B=A^{\Ttran}A+\epsilon D$ by \[B_{R}=P_{R}^{-1}A^{\Ttran}AP_{R}^{-\Ttran}+\epsilon D \] within the Lanczos method. This transformation improves the effective condition number to $\kappa = \norm{(AP_{R}^{-\Ttran})(AP_{R}^{-\Ttran})^{\dagger}}$. However, unlike inner preconditioning, this method necessitates the application of both $P_{R}^{-1}$ and $P_{R}^{-\Ttran}$ in a single Lanczos step. Moreover, an additional step of $P_{R}^{-\Ttran}V$ is necessary to post-process the output because $V$ approximates the null space of $AP_{R}^{-\Ttran}$ instead of $A$.

\section{Numerical experiments}
\label{Sec:NumExp}

In this section, we present some numerical results to support the theoretical results and demonstrate the efficiency of our method. 
\edit{Throughout this section, we assume that the nullity $N$ is unknown prior to computation.}
All numerical experiments\footnote{Scripts to reproduce
numerical results are publicly available at \url{https://github.com/nShao678/Nullspace-code}} in this section are implemented in Matlab~2022b and executed with an AMD Ryzen 9 6900HX Processor (8 cores, 3.3--4.9 GHz) and 32 GB of RAM.
\subsection{Benchmark examples} \label{sec:benchmark}

For our experiments, we will consider two applications of null space computation: (1) Counting the connected components of an undirected graph and (2) cohomology computation.

For an undirected graph $G$, the null space $V_{G}$ of its graph Laplacian $L_{G}$ contains information about its connectivity. Specifically, the dimension of $V_{G}$ is the number of connected components, and the nodes in each connected component can be computed from a null space basis. As a specific example for a graph, we consider the mathematicians' collaboration network from MathSciNet~\cite{Rossi2015}. 
As this graph is connected, the null space of $L_{G}$ is trivial. 
To create a more interesting scenario, we consider the sub-graph $G_{0}$ obtained by first deleting nodes of degree less than $N_{0}$
for a fixed positive integer $N_{0}$ (along with the corresponding edges) and then deleting isolated nodes.
\cref{tabLap} provides statistics in the resulting graphs for $N_{0}=3$ and $N_{0}=10$. As the graph Laplacian $A$ is symmetric positive semi-definite, we will replace $B=A^{\Ttran}A+\epsilon D$ by $B=A+\epsilon D$ in line 1 of~\cref{algoad} for these examples. Moreover, in order to accelerate convergence, we use outer preconditioning (see \cref{sec:outerprecond}) with the preconditioner $P_R$ returned by the incomplete Cholesky decomposition of $A$, using the Matlab function 
\begin{equation}
    \label{defichol}
    P_R=\mathtt{ichol}(A, \mathtt{struct}\text{(`type',`ict',`droptol',$10^{-3}$,`diagcomp',$0.1$)}).
\end{equation}
For $N_{0}=10$, this improves the effective condition number $\kappa=\overline{\sigma}/\underline{\sigma}$ from $6227.1$ to $23.2$.

\begin{table}[htbp]
    \centering
    \caption{Statistics on graph Laplacians for MathSciNet collaboration network after preprocessing.}
    \label{tabLap}
    \begin{tabular}{|c|c|c|c|}
        \hline
        $N_{0}$ & \#nodes & \#connected components & \edit{$\underline{\sigma}$}\\ \hline 
        3 & 168692 & 122 &\edit{0.027} \\ \hline 
        10 & 39890 & 66 & \edit{0.041}\\ \hline 
    \end{tabular}
\end{table}

For cohomology computations, we select the matrices from the SuiteSparse collection~\cite{Davis2011} listed in \cref{tabSSmatrices}.
\begin{table}[htbp]
    \centering
    \caption{SuiteSparse collection matrices related to cohomology computations.}
    \label{tabSSmatrices}
    \begin{tabular}{|c|c|c|c|c|c|c|}
        \hline
        Name & \#rows & \#columns & nnz & type & nullity & \edit{$\underline{\sigma}^{2}$} \\ \hline 
        Franz7 & 10164 & 1740 & 40424 & Rings & 113 &\edit{5.5}\\ \hline
        Franz8 & 16728 & 7176 & 100368 & Rings & 1713 &\edit{2.8}\\ \hline
        Franz9 & 19588 & 4164 & 97508 & Rings & 619 &\edit{8.3}\\ \hline
        Franz10 & 19588 & 4164 & 97508 & Rings & 619 &\edit{8.9}\\  \hline
        GL7d12 & 8899 & 1019 & 37519 & $\mathrm{GL}_{7}(\mathbb{Z})$ & 59&\edit{3.0}\\ \hline 
        GL7d13 & 47271 & 8899 & 356232 & $\mathrm{GL}_{7}(\mathbb{Z})$ & 961&\edit{1.3}\\ \hline
        GL7d14 & 171375& 47271  & 1831183 & $\mathrm{GL}_{7}(\mathbb{Z})$ & 7939&\edit{2.9}\\ \hline
    \end{tabular}
\end{table}

\subsection{Single-vector Lanczos}

We first assess the convergence of the single-vector Lanczos method with full reorthogonalization and without restarting for the graph Laplacian with $N_{0}=10$ from \cref{sec:benchmark}, utilizing the incomplete Cholesky preconditioner~\cref{defichol}.

We varied the perturbation level $\epsilon$ from $4^{-4}$ to $4^{-10}$ and ran single-vector Lanczos method for $\ell=1000$ iterations to ensure convergence. 
\redit{We recall that symmetric positive semi-definiteness of the graph Laplacian allows us to use $B=A+\epsilon D$ instead of $B=A^{\Ttran}A+\epsilon D$ in \cref{algo}. As a consequence, the perturbation result in \cref{thmVANg} becomes $\norm{V_{N}^{\Ttran}AV_{N}}\leq \epsilon^{2}/(\underline{\sigma}^{2}-\epsilon)$.}
\cref{tabConNull} contains $\lambda_{N}(T_{\ell})$, the $N$th smallest eigenvalue of $T_{\ell}$, and norms $\norm{V^{\Ttran}AV}$ and $\norm{AV}$. As expected, $\lambda_{N}(T_{\ell})$ and $\norm{V^{\Ttran}AV}$ are observed to be of the order $\order(\epsilon)$ and $\order(\epsilon^{2})$, respectively.

\begin{table}[htbp]
    \centering
    \caption{Convergence of null space, where $\epsilon$ is taken as $4^{-k}$ for $k=4,\dotsc,10$.}
    \label{tabConNull}
    \begin{tabular}{|c|c|c|c|c|c|c|c|}
        \hline
        $\epsilon$ & 3.9e-03 & 9.8e-04 & 2.4e-04 & 6.1e-05 & 1.5e-05 & 3.8e-06 & 9.5e-07 \\ \hline 
        $\lambda_{N}(T_{\ell})$ & 3.7e-03 & 9.2e-04 & 2.3e-04 & 5.8e-05 & 1.4e-05 & 3.6e-06 & 9.0e-07 \\ \hline 
        $\norm{AV}$ &4.7e-03 & 1.2e-03 & 2.9e-04 & 7.4e-05 & 1.8e-05 & 4.6e-06 & 1.1e-06 \\ \hline
        $\norm{V^{\Ttran}AV}$ &1.5e-06 & 9.2e-08 & 5.7e-09 & 3.6e-10 & 2.2e-11 & 1.4e-12 & 8.7e-14 \\ \hline
    \end{tabular}
\end{table}

The result of \cref{thmmain} predicts that the Ritz values of the single-vector Lanczos method converge individually to the perturbed zero eigenvalues, that is, some converge faster than others.. The numerical results shown in \cref{figConHis} for $\epsilon=4^{-10}$ confirm this result, showing that the number of converged Ritz values grows steadily, a Ritz value converges approximately every $10$ steps of the Lanczos method.

\begin{figure}[htbp]
    \centering
    \includegraphics[width=\figsize]{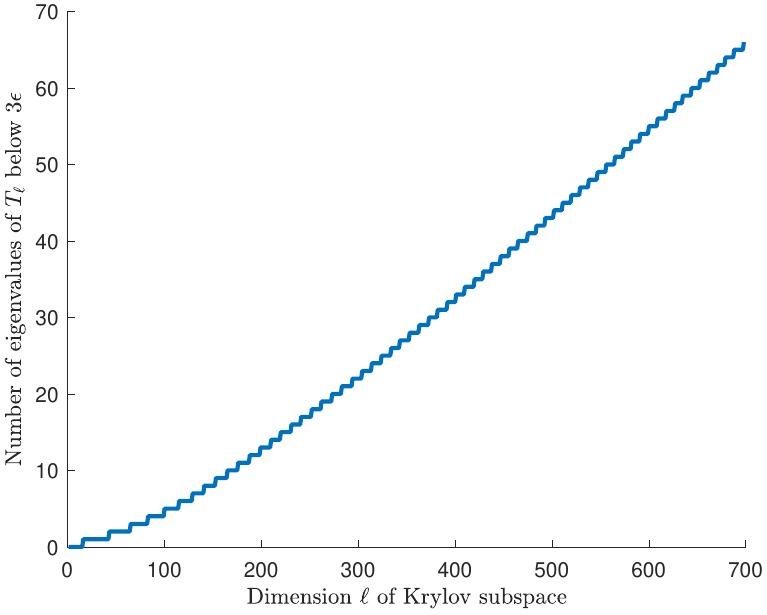}
    \caption{Convergence history of single-vector Lanczos with $\epsilon=4^{-10}$ for graph Laplacian: Number of converged Ritz values vs. number of \edit{matrix-vector multiplications}.}
    \label{figConHis}
\end{figure}

\subsection{Small-block Lanczos method with restarting and partial reorthogonalization}

We now investigate the performance of \cref{algoad}, the small-block Lanczos method with  restarting and partial reorthogonalization. For all remaining experiments, we set the perturbation parameter to $\epsilon=10^{-3}$. 

\subsubsection{Comparison with \texttt{null}}

First, we compare \cref{algoad} with the Matlab built-in function \texttt{null}, which is based on the SVD, for the first six matrices from \cref{tabSSmatrices}. 
The obtained numerical results are presented in \cref{tabCH}.
For all parameter choices, the computed nullity matches the actual nullity from \cref{tabSSmatrices}, \edit{and the approximate basis $V\in\R^{n\times N}$ has almost orthonormal columns in the sense of $\norm{V^{\Ttran}V-I}_{\fro}\leq nN\mathbf{u}$, where $\mathbf{u}\approx 2.22\times 10^{-16}$ is the machine precision.}
It can be observed that restarting is very effective at limiting the dimension of the Krylov subspace.
Although increasing the block size increases the number of matrix-vector multiplications, this increase is, for smaller $d$, outpaced by the higher efficiency of block matrix-vector 
multiplications.
Once the block size becomes too large, both the iteration numbers and execution times start increasing.
Based on our experiments, a suitable strategy for choosing the block size is $d=\dim \mathcal{K}/16$, where $\dim \mathcal{K}$ is the maximum dimension of the Krylov subspace before restarting. \cref{algoad} is faster than \texttt{null} for all examples, and the difference becomes more significant as the matrix size increases. \redit{Although this clearly demonstrates the potential of \cref{algoad}, the reductions in execution time relative to \texttt{null} are not as dramatic as one might hope; computing larger-dimensional null spaces remains a challenging problem and further improvements are needed.}

\begin{table}[htbp]
    \centering
    \caption{Number of matrix vector multiplications and execution time of \cref{algoad} applied to cohomology computations. $\dim \mathcal{K}$ is the maximum dimension of Krylov subspace before restarting. The last column refers to Matlab's built-in function \texttt{null}.}
    \label{tabCH}
    \subfloat[Franz7 with nullity $113$ and $\dim\mathcal{K}=256$.]{
    \begin{tabular}{|c|c|c|c|c|c|c|}
        \hline
        Block size & $d=2$ & $d=4$ & $d=8$ & $d=16$ & $d=32$ & \texttt{null} \\ \hline 
        $\#$matvec & 2520 & 2616 & 2824 & 3232 & 3904& $\times $\\ \hline 
        Time (s) &0.9023 & 0.7106 & 0.5412 & 0.6125 & 0.6515 & 1.5449 \\ \hline 
    \end{tabular}
    }

    \subfloat[Franz8 with nullity $1713$ and $\dim\mathcal{K}=3072$.]{
    \begin{tabular}{|c|c|c|c|c|c|c|}
        \hline
        Block size  & $d=16$ & $d=32$ & $d=64$ & $d=128$ & $d=256$ & \texttt{null}\\ \hline 
        $\#$matvec &22848 & 22880 & 23808 & 25216 & 26880 & $\times $\\ \hline
        Time (s) &96.83 & 77.53 & 71.26 & 69.55 & 77.29 & 82.47 \\ \hline 
    \end{tabular}
    }

    \subfloat[Franz9 with nullity $619$ and $\dim\mathcal{K}=1024$.]{
    \begin{tabular}{|c|c|c|c|c|c|c|}
        \hline
        Block size & $d=8$ & $d=16$ & $d=32$ & $d=64$ & $d=128$ & \texttt{null} \\ \hline 
        $\#$matvec &12496 & 13008 & 12960 & 14528 & 16000 &$\times $ \\ \hline
        Time (s) &17.68 & 13.63 & 11.11 & 9.852 & 9.572 & 21.47 \\ \hline
    \end{tabular}
    }

    \subfloat[Franz10 with nullity $619$ and $\dim\mathcal{K}=1024$.]{
    \begin{tabular}{|c|c|c|c|c|c|c|}
        \hline
        Block size & $d=8$ & $d=16$ & $d=32$ & $d=64$ & $d=128$ & \texttt{null} \\ \hline 
        $\#$matvec &12440 & 12784 & 12960 & 14464 & 16000 &$\times $\\ \hline
        Time (s) &17.54 & 12.81 & 10.43 & 9.74 & 9.028 & 21.15\\ \hline
    \end{tabular}
    
    }

    \subfloat[GL7d12 with nullity $59$ and $\dim\mathcal{K}=128$.]{
        \begin{tabular}{|c|c|c|c|c|c|c|}
            \hline
            Block size & $d=1$ & $d=2$ & $d=4$ & $d=8$ & $d=16$ & \texttt{null} \\ \hline 
            $\#$matvec &1439 & 1482 & 1504 & 1720 & 1920  &$\times $ \\ \hline
            Time (s) &0.2188 & 0.1858 & 0.1545 & 0.1752 & 0.2387 & 0.6476 \\ \hline
        \end{tabular}
    }
    
    \subfloat[GL7d13 with nullity $961$ and $\dim\mathcal{K}=2048$.]{
        \begin{tabular}{|c|c|c|c|c|c|c|}
            \hline
            Block size & $d=16$ & $d=32$ & $d=64$ & $d=128$ & $d=256$& \texttt{null} \\ \hline 
            $\#$matvec &17552 & 18176 & 19072 & 19712 & 23040 & $\times $\\ \hline
            Time (s) &55.20 & 44.71 & 38.81 & 37.29 & 41.50 & 212.2 \\ \hline
        \end{tabular}
    }

\end{table}

\subsubsection{Comparison with large-block Lanczos}

Traditionally, the block Lanczos method with restarting is used with a block size  that is  slightly larger than the nullity. Obviously, when memory is a constraint, the small-block or even single-vector Lanczos methods with restarting are preferred. Even when memory is not a constraint, using the small-block version can be beneficial.
To illustrate this, we consider two examples: the matrix ``Franz9'' from cohomology computations and the graph Laplacian with $N_{0}=3$, which have nullities $619$ and $122$, respectively. We compare the performance of the small-block and large-block Lanczos methods. For the large-block Lanczos method, we set the maximum dimension of the Krylov subspace before restarting to be approximately four times the nullity, \ie $2560$ and $512$, respectively.

As shown in \cref{tabCHLB}, the large-block Lanczos method requires more matrix-vector multiplications and execution time compared to the small-block version. The nullity has been computed correctly by all methods. One advantage of the large-block Lanczos method is that it does not require small perturbations such as $B=A^{\Ttran}A+\epsilon D$ or $B=A+\epsilon D$, thanks to the use of a sufficiently large-block size. While the absence of perturbations yields more accurate results, the overall computational cost remains unaffected because the additional computational cost from applying the diagonal perturbation matrix $D$ is negligible.

\begin{table}[htbp]
    \centering
    \caption{Number of matrix vector multiplications, execution time, and null space residual of~\cref{algoad} applied to ``Franz9'' and the graph Laplacian with $N_{0}=3$. $\dim \mathcal{K}$ is the maximum dimension of Krylov subspace before restarting. The last column refers to the block Lanczos method applied to the original matrix (without diagonal perturbations). The algorithm is stopped once $\norm{AV}\leq 10^{-4}$ for ``Franz9'' or $\norm{AV}\leq 1.2\times 10^{-3}$ for the graph Laplacian is detected before restarting.}
    \label{tabCHLB}
    \subfloat[Franz9 with nullity $619$ and $\dim\mathcal{K}=2560$.]{
    \begin{tabular}{|c|c|c|c|c|c|c|}
        \hline
        &\multicolumn{5}{|c|}{Small-block Lanczos}&Block Lanczos\\ \hline
        Block size & $d=40$ & $d=80$ & $d=160$ & $d=320$ & $d=640$ & $d=640$   \\ \hline 
        $\#$matvec & 7960 & 6960 & 7200 & 8960 & 10240 & 10240   \\ \hline
        Time (s) &13.57 & 10.27 & 9.925 & 12.32 & 11.23 & 11.38  \\ \hline
        $\norm{AV}$& 9.6e-05 & 9.6e-05 & 1.0e-04 & 9.6e-05 & 9.6e-05 & 1.2e-05\\ \hline
    \end{tabular}}

\subfloat[Graph Laplacian with nullity $122$ and $\dim \mathcal{K}=512$.]{
    \begin{tabular}{|c|c|c|c|c|c|c|}
        \hline
        &\multicolumn{5}{|c|}{Small-block Lanczos}&Block Lanczos\\ \hline
        Block size & $d=8$ & $d=16$ & $d=32$ & $d=64$ & $d=128$ & $d=128$  \\ \hline 
        $\#$matvec & 1608 & 1632 & 1664 & 2048 & 3584 & 3584   \\ \hline
        Time (s) &48.55 & 43.52 & 40.29 & 44.42 & 70.72 & 69.70  \\ \hline
        $\norm{AV}$& 1.0e-03 & 1.0e-03 & 1.0e-03 & 1.1e-03 & 1.0e-03 & 7.5e-04\\ \hline
    \end{tabular}}

\end{table}

\subsubsection{Large-scale problems}

In the final experiment, we address a particularly challenging problem where both the dense method and the traditional block Lanczos method are infeasible due to memory constraints. For the matrix ``GL7d14'', it is impossible to apply the \texttt{null} function because the dense matrix cannot be stored in memory. Similarly, the large-block Lanczos method encounters a memory bottleneck, making it impractical to perform a restart after even just two iterations. \cref{tabCHLS} demonstrates that the the small-block Lanczos method successfully computes the null space within a Krylov subspace of dimension 10240. Remarkably, this dimension is only 1.29 times larger than the nullity, highlighting the efficiency and feasibility of the small-block method in handling such large-scale problems.

\begin{table}[htbp]
    \centering
    \caption{Number of matrix vector multiplications and execution time of \cref{algoad} applied to ``GL7d14'' with nullity $7939$. The maximum dimension of Krylov subspace before restarting is $10240$. The computed nullity is correct and $\norm{AV}\approx1.4\times10^{-4}$}
    \label{tabCHLS}
    \begin{tabular}{|c|c|c|c|c|c|c|c|}
        \hline
        Block size & $d=16$ & $d=32$ & $d=64$ & $d=128$ & $d=256$ & $d=512$   \\ \hline 
        $\#$matvec & 133456 & 133344 & 132288 & 132352 & 133888 & 140800   \\ \hline
        Time(s) &11516 & 8137.4 & 6986.4 & 6375.0 & 6218.7 & 6716.4 \\ \hline
    \end{tabular}    
\end{table}

\section{Conclusion}
In this work, we proposed a novel randomized small-block Lanczos method that incorporates modern Krylov subspace techniques such as restarting and partial reorthogonalization to compute null spaces of large sparse matrices. For the examples considered, our implementation outperforms both the default SVD-based null space solver and the traditional block Lanczos method. Additionally, we provide a theoretical analysis of the convergence for the single-vector case, that is, $d = 1$. It remains an open problem to derive convergence bounds that reflect the good performance observed for small $d>1$.

\section*{Acknowledgments}
The authors thank Yuji Nakatsukasa and Rikhav Shah for helpful discussions on this work. \edit{They also thank the two anonymous
referees, who provided useful comments and critical insights.}

\end{document}